\documentclass[final, 11pt,a4paper]{amsart} 

\usepackage[all, 2cell]{xy}
\UseAllTwocells
\usepackage{placeins}
\usepackage{enumitem}
\usepackage{amssymb}
\usepackage{latexsym}
\usepackage{amsmath}
\usepackage{mathrsfs}
\usepackage{array,booktabs}
\usepackage{verbatim}
\usepackage{amsthm}
\usepackage{dsfont}
\usepackage{bbm}

\usepackage[dvipsnames]{xcolor}

\usepackage{fullpage}

\theoremstyle{plain}
\newtheorem{theorem}{Theorem}[section]
\newtheorem*{theorem*}{Theorem}

\newtheorem{lemma}[theorem]{Lemma}
\newtheorem{corollary}[theorem]{Corollary}
\newtheorem{proposition}[theorem]{Proposition}

\theoremstyle{definition}
\newtheorem{remark}[theorem]{Remark}

\newtheorem{example}[theorem]{Example}

\newtheorem{construction}[theorem]{Construction}
\newtheorem*{acknowledgements}{Acknowledgements}

\newcommand{\Z}{\mathbb{Z}}
\newcommand{\Q}{\mathbb{Q}}
\renewcommand{\P}{\mathbb{P}}
\newcommand{\T}{\mathcal{T}}

\newcommand{\Tc}{\mathcal{T}^\mathrm{c}}
\renewcommand{\c}{{\mathrm{c}}}
\newcommand{\heart}[1]{\mathcal{G}_{#1}}
\newcommand{\Mod}[1]{\mathrm{Mod}\text{-}{#1}}
\renewcommand{\mod}[1]{\mathrm{mod}\text{-}{#1}}
\newcommand{\Cogen}[1]{\mathrm{Cogen}(#1)}
\newcommand{\y}{\mathbf{y}}
\newcommand{\Prod}[1]{\mathrm{Prod}(#1)}
\newcommand{\ele}[1]{\mathrm{Cogen}_*(#1)}
\newcommand{\Def}[2]{\mathrm{Def}_{#1}({#2})}
\newcommand{\Ab}{\mathbf{Ab}}
\newcommand{\Hom}[1]{\mathrm{Hom}_{#1}}
\newcommand{\Ext}[1]{\mathrm{Ext}_{#1}^1}
\newcommand{\Pinj}[1]{\mathrm{Pinj}(#1)}
\newcommand{\Inj}[1]{\mathrm{Inj}(#1)}
\newcommand{\Fpinj}[1]{\mathrm{Fpinj}(#1)}
\newcommand{\D}[1]{\mathrm{D}(#1)}
\newcommand{\fp}[1]{#1^{\mathrm{fp}}}
\newcommand{\DModR}{\mathrm{D}(\Mod{R})}
\newcommand{\Ch}[1]{\mathrm{Ch}(#1)}
\newcommand{\Der}{\mathbb{D}}
\newcommand{\hocolim}[2]{\mathrm{hocolim}_{#1}(#2)}
\newcommand{\holim}[2]{\mathrm{holim}_{#1}(#2)}
\newcommand{\hocolimf}[1]{\mathrm{hocolim}_{#1}}
\newcommand{\holimf}[1]{\mathrm{holim}_{#1}}
\newcommand{\CAT}{\mathcal{C}AT}
\newcommand{\Cat}{\mathcal{C}at}
\newcommand{\op}{{\mathrm{op}}}

\newcommand{\base}{\mathbf{1}}
\newcommand{\dia}[1]{\mathrm{dia}_{#1}}
\newcommand{\Pow}[1]{\mathcal{P}(#1)}
\newcommand{\Red}[1]{\mathrm{Red}_\mathcal{F}(#1)}
\newcommand{\Redf}{\mathrm{Red}_\mathcal{F}}
\newcommand{\colim}{\mathop {\rm colim}}
\newcommand{\kk}{k}
\newcommand{\id}{\mathrm{id}}
\newcommand{\Redu}[1]{#1_\mathcal{F}}
\renewcommand{\ker}[1]{\mathrm{ker}#1}
\newcommand{\im}[1]{\mathrm{im}#1}

\numberwithin{equation}{section}

\begin{document}
\title{Purity in compactly generated derivators and t-structures with Grothendieck hearts}
\date{}
\author{Rosanna Laking}
\address{Dipartimento di Informatica, Settore di Matematica, Universit\'a degli Studi di Verona, Strada le Grazie 15, Ca' vignal 2, I-37134 Verona, Italy }
\email{rosanna.laking@univr.it}
\thanks{The author was supported by the DFG SFB / Transregio 45 and by the Max Planck Institute for Mathematics.}
\maketitle

\begin{abstract}
We study t-structures with Grothendieck hearts on compactly generated triangulated categories $\T$ that are underlying categories of strong and stable derivators.  This setting includes all algebraic compactly generated triangulated categories. We give an intrinsic characterisation of pure triangles and the definable subcategories of $\T$ in terms of directed homotopy colimits.  For a left nondegenerate t-structure ${\bf t}=(\mathcal{U},\mathcal{V})$ on $\T$, we show that $\mathcal{V}$ is definable if and only if ${\bf t}$  is smashing and has a Grothendieck heart.  Moreover, these conditions are equivalent to ${\bf t}$ being homotopically smashing and to ${\bf t}$ being cogenerated by a pure-injective partial cosilting object.  Finally, we show that finiteness conditions on the heart of ${\bf t}$ are determined by purity conditions on the associated partial cosilting object.
\end{abstract}

\section{Introduction.}

Triangulated categories arising in representation theory, algebraic geometry and topology often come with the additional structure of a t-structure \cite{BBD}, allowing one to carry out homological algebra with respect to this t-structure.  A t-structure ${\bf t}$ on a triangulated category $\T$ is a torsion pair satisfying additional properties ensuring that there exists an abelian subcategory $\mathcal{G}$ of $\T$, called the heart, and a cohomological functor from $\T$ to $\mathcal{G}$.  

The question of when the heart of a t-structure is a Grothendieck category is a natural one and has been pursued by many authors using a variety of techniques.  Given the scarcity of limits and colimits in $\T$, a necessary ingredient in solving this problem is a method for understanding how directed colimits might look in $\mathcal{G}$.  The various approaches to this problem tend to follow one of two general strategies:\begin{description}
\item[Strategy 1] Consider $\T$ as a subcategory of a Grothendieck category $\mathcal{A}(\T)$ and understand directed colimits in $\mathcal{G}$ in terms of directed colimits in $\mathcal{A}(\T)$ (\cite{AMV}, \cite{Bazz}, \cite{MV} and \cite{MR3239134}).
\item[Strategy 2]  Consider $\T$ as the underlying category of some higher categorical structure and understand directed colimits in $\mathcal{G}$ in terms of directed homotopy colimits (\cite{HA}, \cite{SSV}).
\end{description}

Strategy 1 is most effective in the setting where $\T$ is generated by its subcategory of compact objects $\Tc$. In this case $\mathcal{A}(\T)$ is the category $\Mod{\Tc}$ and we have the well-developed theory of purity available to us  (see, for example, \cite{MR1728877, MR1754234}).    

In this paper we combine Strategies 1 and 2: we consider the case where $\T$ is compactly generated and is also the underlying category of a strong and stable derivator.  By \cite[Thm.~1.36]{Der} and \cite{Keller}, this includes all algebraic compactly generated triangulated categories.  We observe that, in this setting, the strategies above are completely compatible: the image in $\Mod{\Tc}$ of each directed homotopy colimit corresponds to an appropriate directed colimit in $\Mod{\Tc}$ (see Remark \ref{Rem: hocolim to dir lim}).  

From this starting point, we are able to characterise important notions from the theory of purity in terms of certain homotopy colimits, which we call coherent ultrapowers, inspired by their use in model theory; see Section \ref{Sec: Red Prod} for details.  These results are analogous to well-known characterisations of purity in module categories (see, for example, \cite[Thm.~4.2.18, Thm.~16.1.16]{MR2530988} and \cite[Sec.~2.3]{MR1648602}) and generalise \cite[Thm.~7.5]{MR1899046}.  

\begin{theorem*}[{Proposition \ref{Prop: pure tri}, Theorem \ref{Thm: def cats}}] Let $\T$ be a compactly generated triangulated category and suppose that $\T$ is the underlying category of a strong and stable derivator. Then the following statements hold.\begin{enumerate}
\item A triangle $\delta \colon X \rightarrow Y \rightarrow Z \rightarrow \Sigma X$ is a pure triangle if and only if there is some coherent ultrapower of $\delta$ that is a split triangle.
\item A full subcategory of $\T$ is definable if and only if it is closed under direct products, directed homotopy colimits and pure subobjects.
\end{enumerate}
\end{theorem*}

Using the interaction between purity and homotopy colimits, we are able to show that the natural class of (left nondegenerate) t-structures with Grothendieck hearts considered via Strategy 1 (those cogenerated by pure-injective cosilting objects) coincides with the class considered via Strategy 2 (homotopically smashing t-structures).  That is, we introduce the notion of partial cosilting t-structures and prove the following theorem (which specialises to the case of nondegenerate and cosilting t-structures).

\begin{theorem*}[{Theorem \ref{Thm: groth hearts}}]
Let $\T$ be a compactly generated triangulated category and suppose that $\T$ is the underlying category of a strong and stable derivator.  Let ${\bf t} = (\mathcal{U}, \mathcal{V})$ be a left nondegenerate t-structure on $\T$.  Then the following statements are equivalent: \begin{enumerate}
\item ${\bf t} $ is a partial cosilting t-structure with a pure-injective partial cosilting object $C$.
\item $\mathcal{V}$ is definable.
\item ${\bf t} $ is homotopically smashing.
\item ${\bf t} $ is smashing and the heart $\mathcal{G}$ is a Grothendieck category.
\end{enumerate}
\end{theorem*}

In order to place this work in context, let us give a brief summary of the preceding results concerning Grothendieck hearts.  In \cite{MR3336001, MR3448805} the authors consider t-structures in the derived category $\D{\mathcal{H}}$ of a Grothendieck category $\mathcal{H}$ induced by torsion pairs in $\mathcal{H}$ i.e.~Happel-Reiten-{Smal\o} (HRS) t-structures \cite{HRS}. They show that the heart of a HRS t-structure ${\bf t}$ is Grothendieck if and only if the torsion-free class in $\mathcal{H}$ is closed under directed colimits.  When $\mathcal{H}$ is a module category, these are exactly the torsion-free classes of the form $\Cogen{C}$ for cosilting modules $C$ (see \cite[Cor.~3.9]{abundance}).

The question of when the heart of a t-structure is a Grothendieck category has also been considered in the context of silting theory.  The t-structure induced by a large tilting module $T$ has a Grothendieck heart exactly when $T$ is pure-projective \cite[Thm.~7.5]{Bazz} and the t-structure induced by a large cotilting module always has a Grothendieck heart \cite{MR3239134}.  A more general version of these results can be found in \cite[Thm.~3.6, Thm.~3.7]{AMV} in the context of compactly generated triangulated categories: the authors show that a nondegenerate cosmashing (respectively smashing) t-structure has a Grothendieck heart if and only if it is (co)generated by a pure-projective (respectively pure-injective) silting (respectively cosilting) object.  The result related to cosilting can also be found in \cite[Prop.~2]{NSZ} in the case where the t-structure is assumed to be smashing and cosmashing.  Related to the silting t-structures is the case of compactly generated t-structures; these have been shown to have Grothendieck hearts in various settings, for example the homotopy category of a combinatorial stable model category \cite[Cor.~D]{SSV} \cite[Thm.~0.2]{Bon}.

Finally, in both \cite{SSV} and \cite{HA}, the authors consider t-structures such that the coaisle is closed under homotopy colimits, we shall refer to such t-structures as homotopically smashing.  In \cite[Thm.~B]{SSV}, it is shown that for any (strong and stable) derivator this assumption is enough to ensure the heart has exact directed colimits.  When, in addition, the t-structure is on the homotopy category of a stable combinatorial model category, the authors prove that the heart is a Grothendieck category.  In a similar vein, Lurie considers homotopically smashing t-structures with accessible aisles in the context of presentable stable $\infty$-categories \cite[Rem.~1.3.5.23]{HA} and observes that also this implies that the heart is a Grothendieck category.

The question of when hearts satisfy various finiteness conditions has also been considered in the literature. For example in \cite{MR3613439} HRS t-structures with locally coherent hearts are characterised and in \cite{MR2577659} locally noetherian $1$-cotilting t-structures are shown to be induced by $\Sigma$-pure-injective $1$-cotilting modules. In the final section we consider the case where the heart $\mathcal{G}$ has an internal notion of purity and we investigate how the purity in $\T$ interacts with the purity in $\mathcal{G}$.  As an application, we generalise the above results by characterising finiteness conditions on $\mathcal{G}$ in terms of purity assumptions on the corresponding partial cosilting object.   

\begin{theorem*}[{Proposition \ref{Prop: loc noeth}, Proposition \ref{Prop: ele cogen loc coh}, Theorem \ref{Thm: loc coh ele cogen}}]
Let $\T$ be a compactly generated triangulated category and suppose that $\T$ is the underlying category of a strong and stable derivator. Let ${\bf t} = (\mathcal{U}, \mathcal{V})$ be a partial cosilting t-structure on $\T$ with partial cosilting object $C$.  Then the following statements hold: \begin{enumerate}
\item If $C$ is $\Sigma$-pure-injective, then the heart $\mathcal{G}$ is locally noetherian.
\item If $C$ is an elementary cogenerator, then the heart $\mathcal{G}$ is locally coherent.
\end{enumerate} If $\Prod{C}$ is contained in $\mathcal{G}$, then the converse statements hold. 
\end{theorem*}

Note that Example \ref{ex: counter-example} (communicated to the author by Michal Hrbek) shows that the converse statements need not hold when $\Prod{C}$ is not contained in $\mathcal{G}$. This statement differs from the original version of this article.

We end this introduction with a summary of the content of the paper.  In Section \ref{Sec:der}, we introduce the basic definitions and notation related to the theory of derivators, as well as the definition of the coherent reduced products and coherent ultraproducts (see Section \ref{Sec: Red Prod}).  The construction and proof that coherent reduced products exist is contained in Appendix \ref{App: Proof}.  In Section \ref{Sec: Purity} we consider purity in strong and stable derivators whose underlying category is compactly generated; we prove that pure triangles can be detected using coherent ultraproducts and that definable subcategories can be characterised via closure conditions.  Section \ref{Sec: GrothHearts} is concerned with t-structures whose hearts are Grothendieck categories.  We introduce partial cosilting t-structures and homotopically smashing t-structures and in Theorem \ref{Thm: groth hearts} we show that the left nondegenerate t-structures with these properties coincide.  We end the section with an example of a t-structure satisfying the equivalent statements of Theorem \ref{Thm: groth hearts}.  The final section is dedicated to understanding how purity in the triangulated category relates to purity in the heart.  We use this to characterise finiteness conditions on the heart in terms of properties of the cosilting object.

\begin{acknowledgements}
The author would like to thank Moritz Groth and Gustavo Jasso for many interesting and helpful conversations about derivators and other higher categorical structures.  She would like to thank Lidia Angeleri H\"ugel, Frederik Marks and Jorge Vit\'oria for discussions regarding the definition of partial cosilting, which also led to the contents of Example \ref{Ex: recollement}.  Particular thanks are extended to Prof.~Angeleri H\"ugel for her ongoing support of this project.  

Thank you also to Michal Hrbek for pointing out a mistake in Proposition \ref{Prop: loc noeth} in the original (published) version of this article.  The current version contains the corrected proof and the counter-example to the original statement.
\end{acknowledgements}

\section{Derivators.}\label{Sec:der}

The main concept we use from the theory of derivators is that of homotopy limits and colimits.  Their definition is a generalisation of limits and colimits in a category $\mathcal{C}$. That is, for a small category $A$, the functors $\lim_A$ and $\colim_A$ arise as right and left adjoints to the constant diagram functor $\Delta_A \colon \mathcal{C} \rightarrow \mathcal{C}^A$.  The diagram categories $\mathcal{C}^A$ are the values of the contravariant 2-functor $y_\mathcal{C}$ from the 2-category $\Cat$ of small categories to the 2-category $\CAT$ of all categories defined on objects by $A \mapsto \mathcal{C}^A$.  Accordingly, we define a \textbf{prederivator} to be a 2-functor $\Der \colon \Cat^\op \rightarrow \CAT$, the values of which are referred to as \textbf{coherent diagram categories}.   Moreover, the functors $\Delta_A$ are the values $y_\mathcal{C}(\pi_A)$ where $\pi_A$ is the unique functor from $A$ to the category $\base$ with a single object and only the identity morphism.  A \textbf{derivator} is then defined to be a prederivator $\Der \colon \Cat^\op \rightarrow \CAT$ satisfying the axioms (Der1)-(Der4) (see Section \ref{Sec: axioms}).  Crucially, the axiom (Der3) implies that the functors $\Der(\pi_A)$ always have a right adjoint $\holimf{A}$ and a left adjoint $\hocolimf{A}$.

\subsection{A brief introduction to derivators.} 

In this section we give an overview of the definitions related to derivators that we need in later sections.  In order to make the definitions more concrete, we use the derivator associated to the unbounded derived category of a ring (described below in Example \ref{Ex: derived category}) to illustrate each one.  The axioms (Der1)-(Der4) and an explanation of shifted derivators are contained in Appendix \ref{App: axioms}.

\subsubsection{Basic terminology.}

Throughout the paper we use the following basic terminology and notation for a prederivator $\Der \colon \Cat^\op \rightarrow \CAT$: \begin{itemize}
\item  Let $\base$ denote the category with a single object and its identity morphism.  This is a terminal object in $\Cat$ and, for each small category $A$, we denote the unique arrow $A \rightarrow \base$ by $\pi_A$ (or $\pi$ if it is unambiguous). The category $\Der(\base)$ is referred to as the \textbf{underlying category of $\Der$}.

\item The category $\Der(A)$ for each object $A$ in $\Cat$ is called the \textbf{coherent diagram category} and the objects of $\Der(A)$ are called \textbf{coherent diagrams of shape $A$}. The objects of the category $\Der(\base)^A$ are referred to as the \textbf{incoherent diagrams of shape $A$}.

\item Let $\mathbf{2}$ denote the poset $\{0 < 1\}$ considered as a category.  The category $\Der(\base)^{\mathbf{2}}$ consists of the morphisms in $\Der(\base)$ and is referred to as the \textbf{category of incoherent morphisms}.  The category $\Der(\mathbf{2})$ is referred to as the category of \textbf{coherent morphisms}.

\item The functor $\Der(u)$ for each functor $u \colon A \rightarrow B$ in $\Cat$ is called the \textbf{restriction functor} and is denoted by $u^*$.  Similarly, the natural transformation $\Der(\alpha)$ is denoted by $\alpha^*$ for each natural transformation $\alpha \colon u \Rightarrow v$ in $\Cat$.
\end{itemize} 

\begin{example}\label{Ex: derived category}
Let $R$ be a ring and let $\mathcal{A}$ denote the category $\Mod{R}$ of right $R$ modules.  There is a prederivator $\Der_R \colon \Cat^\op \rightarrow \CAT$ such that: \begin{enumerate}
\item For each small category $A$, the category $\Der_R(A)$ is the unbounded derived category $\D{\mathcal{A}^A}$.

\item For each functor $u \colon A \to B$, the functor $u^* \colon \D{\mathcal{A}^B}\rightarrow \D{\mathcal{A}^A}$ is induced from the exact functor $\mathcal{A}^B \rightarrow \mathcal{A}^A$.  
\end{enumerate}

\noindent This is a 2-functor since every natural transformation $\alpha \colon u \Rightarrow v$ induces a natural transformation $\alpha^* \colon \Der_R(u) \Rightarrow \Der_R(v)$. The underlying category $\Der_R(\base)$ of $\Der_R$ is equivalent to $\DModR$.
\end{example}

For each functor $u \colon A \rightarrow B$, if the restriction functor $u^* \colon \Der(B) \rightarrow \Der(A)$ has a right adjoint, then it is denoted $u_* \colon \Der(A) \rightarrow \Der(B)$.  Similarly, if it has a left adjoint, then it is denoted $u_! \colon \Der(A) \rightarrow \Der(B)$.  These are referred to as right and left \textbf{Kan extensions}.

\subsubsection{Underlying diagram functors.}\label{Sec: dia}

Next we will describe some distinguished restriction functors that will be used frequently in the subsequent sections: \begin{itemize}
\item Let $a$ be an object in a small category $A$ and let $a \colon \base \rightarrow A$ denote the unique functor mapping the object in $\base$ to $a$.  Then the restriction functor $a^* \colon \Der(A) \rightarrow \Der(\base)$ is called the \textbf{evaluation functor at $a$}.  For an object $X$ in $\Der(A)$, the image of $X$ under $a^*$ is called the \textbf{value of $X$ at $a$} and is denoted by $X_a$.  

\item For every morphism $f \colon a \rightarrow b$ in $A$, let $f \colon a\Rightarrow b$ denote the natural transformation from $a$ to $b$ with the unique component given by $f$.  For each object $X$ in $\Der(A)$ we denote the component of $f^* \colon a^* \Rightarrow b^*$ at $X$ by $X_f \colon X_a \to X_b$; this is called the \textbf{value of $X$ at $f$}.
\end{itemize}

We define the \textbf{underlying diagram functor} $\dia{A} \colon \Der(A) \rightarrow \Der(\base)^A$ for each small category $A$ in the following way: \begin{itemize}
\item For each object $X$ in $\Der(A)$ we assign the object $\dia{A}(X) \colon A \rightarrow \Der(\base)$ of $\Der(\base)^A$ such that $a \mapsto X_a$ and $f \mapsto X_f$.
\item For each morphism $g \colon X \rightarrow Y$ in $\Der(A)$ we assign the morphism $\dia{A}(g) \colon \dia{A}(X) \rightarrow \dia{A}(Y)$ in $\Der(\base)^A$ given by $\dia{A}(g)_a := g_a \colon X_a \rightarrow Y_a$ for each object $a$ in $A$.
\end{itemize}

\begin{example}\label{Ex: dia R} Let $X$ be an object of $\Der_R(A)$ for a small category $A$.  We may consider an object $X'$ in $\Ch{\mathcal{A}^A} \simeq \Ch{\mathcal{A}}^A$ that maps to $X$ under the localisation functor.  Consider $X'$ as an $A$-shaped diagram in $\Ch{\mathcal{A}}^A$, we may postcompose with the localisation functor to obtain an object of $\Der_R(\base)^A$.  The assignment is well-defined and extends to the functor $\dia{A}\colon \Der_R(A) \rightarrow \Der_R(\base)^A$.

In general, the category $\Der_R(\base)^A$ of incoherent diagrams is not equivalent to the category $\Der_R(A)$ of coherent diagrams.  For example, if $\kk$ is a field then $\Der_\kk(\base)$ is abelian and hence $\Der_\kk(\base)^A$ is abelian for any $A$.  In contrast, $\Der_\kk(\mathbf{2})$ is the derived category of $\kk$-representations of the quiver $\bullet \rightarrow \bullet$, which is clearly not abelian.
\end{example}

\subsubsection{Homotopy limits and homotopy colimits.}\label{Sec: holim}

For every small category $A$ consider the unique functor $\pi = \pi_A \colon A \rightarrow \base$.  We denote the left adjoint $\pi_!$ of $\pi^*$ by $\mathrm{hocolim}_{A} \colon \Der(A) \rightarrow \Der(\base)$ and refer to it as the \textbf{homotopy colimit functor}.  Similarly, we denote the right adjoint $\pi_*$ of $\pi^*$ by $\mathrm{holim}_{A} \colon \Der(A) \rightarrow \Der(\base)$ and refer to it as the \textbf{homotopy limit functor}.  For an object $X$ in $\Der(A)$, we refer to $\hocolim{A}{X}$ as the \textbf{homotopy colimit of $X$} and to $\holim{A}{X}$ as the \textbf{homotopy limit of $X$}.

\begin{example}\label{Ex: lim is lim}
Consider a directed category $I$ and the unique functor $\pi \colon I \rightarrow \base$. Following \cite[Prop.~6.6]{MR3239134}, we can describe the left Kan extension $(\pi_I)_!$ with respect to $\Der_R$ explicitly. The colimit functor $\varinjlim \colon \Ch{R}^I \rightarrow \Ch{R}$ and the constant diagram functor $\Delta_I \colon \Ch{R} \rightarrow \Ch{R}^I$ form an adjoint pair of exact functors.  Hence they induce a pair of adjoint functors $\varinjlim \colon \Der_R(I) \rightarrow \Der_R(\base)$ and $\pi^* \colon \Der_R(\base) \rightarrow \Der_R(I)$.  Left adjoint functors are unique up to equivalence and so we conclude that $\varinjlim_I \cong \hocolimf{I}$.    

As products are exact in $\Mod{R}$, similar reasoning may be applied to the direct product functor on $\Ch{R}$ to obtain an example of a homotopy limit functor.
\end{example}

\subsubsection{Strong and stable derivators.}

For small categories $A$ and $B$, we define the \textbf{partial underlying diagram functor} \[\dia{B,A} \colon \Der(A\times B) \rightarrow \Der(A)^B\] to be the underlying diagram functor $\dia{B}^{\Der^A} \colon \Der^A(B) \rightarrow \Der^A(\base)^{B}$ with respect to the shifted derivator $\Der^A$ (which is defined in Section \ref{Sec: shifted}).  A derivator is called \textbf{strong} if the partial underlying diagram functor $\dia{F,A}$ is full and essentially surjective for every small category $A$ and every finite free category $F$. 

\begin{example}
The derivator $\Der_R$ is strong and so the functor $\dia{\mathbf{2}} \colon \Der_R(\mathbf{2}) \rightarrow \Der_R(\base)^{\mathbf{2}}$, described in Example \ref{Ex: dia R}, is full and essentially surjective. We may therefore replace any incoherent morphism with the underlying diagram of a coherent morphism.  
\end{example}
Another important property of $\Der_R$ is that, for every small category $A$, the category $\Der_R(A)$ is a triangulated category.  This is a consequence of $\Der_R$ being \textbf{stable} as well as strong.  In this article we do not give the definition of a stable derivator; instead, we refer the reader to \cite{Der}.  In \cite[Sec.~4.2]{Der} both a canonical endofunctor $\Sigma_A^\Der \colon \Der(A) \to \Der(A)$ and a class of distinguished triangles $\triangle_A^\Der$ is defined for each strong stable derivator $\Der$ and each small category $A$.  It turns out that this defines a canonical triangulated structure on $\Der(A)$.  

Another particularly important property of $\Der$ is that the functors $u^*$, $u_*$, $u_!$ preserve the canonical triangulated for each $u \colon A \to B$.  We call an additive functor $F \colon \T \rightarrow \mathcal{T}'$ between triangulated categories \textbf{exact} if there exists a natural isomorphism $\eta \colon F \circ \Sigma \rightarrow \Sigma \circ F$ and for every triangle $X \overset{f}{\rightarrow} Y \overset{g}{\rightarrow} Z \overset{h}{\rightarrow} \Sigma X$ in $\T$ we have that \[FX \overset{Ff}{\longrightarrow} FY \overset{Fg}{\longrightarrow} FZ \overset{\eta_X \circ Fh}{\longrightarrow} \Sigma(FX)\] is a triangle in $\mathcal{T}'$.  

We summarise the above discussion in the following proposition.  The first statement can be found in \cite[Thm.~4.16]{Der} and the second can be found in \cite[Sec.~10]{Groth}.

\begin{proposition} \label{Prop: strong and stable}
Let $\Der$ be a strong and stable derivator.  Then the following statements hold. \begin{enumerate}
\item For any small category $A$, the category $\Der(A)$ has a canonical triangulated structure given by the suspension functor $\Sigma_A^\Der$ and class of distinguished triangles $\triangle_A^\Der$.
\item For any functor $u \colon A \rightarrow B$ in $\Cat$, the functors $u^*$, $u_*$ and $u_!$ are exact functors with respect to the canonical triangulated structure.
\end{enumerate}
\end{proposition}

\subsection{Coherent reduced products.}\label{Sec: Red Prod}

In Section \ref{Sec: Purity} we give intrinsic characterisations of both pure-exact triangles and definable subcategories in the underlying category $\Der(\base)$ of a derivator $\Der$ (assuming that $\Der(\base)$ is compactly generated).  These characterisations mimic those given for a locally finitely presented category $\mathcal{C}$ in \cite[Thm.~16.1.16]{MR2530988} and \cite[Cor.~4.6]{exactly}.  A key tool in both of these cases is the reduced product with respect to a proper filter (see Construction \ref{Con: reduced product}).  In this section we show that, for any derivator $\Der$ and any proper filter, there is a coherent diagram whose underlying diagram is isomorphic to the direct system described in Construction \ref{Con: reduced product}.

Let $S$ be a non-empty set and let $\Pow{S}$ denote the power set of $S$.  Then a \textbf{(proper) filter on $S$} is a non-empty collection $\mathcal{F} \subset \Pow{S}$ of subsets of $S$ (not containing $\emptyset$) such that if $P, Q \in \mathcal{F}$ then $P\cap Q \in \mathcal{F}$ and if $P \in \mathcal{F}$ and $P \subseteq Q$ then $Q \in \mathcal{F}$.  Moreover, a proper filter $\mathcal{F}$ is called an \textbf{ultrafilter} if, for every $P \in \Pow{S}$, either $P \in \mathcal{F}$ or $S\setminus P \in \mathcal{F}$.

\begin{construction}\label{Con: reduced product}
Let $\mathcal{C}$ be a complete category with directed colimits and let $X = \{X_s\}_{s\in S}$ be a set of objects in $\mathcal{C}$.  Given a proper filter $\mathcal{F}$ on $S$, we may define a directed system in $\mathcal{C}$ consisting of objects $\{ \prod_{p\in P} X_p \mid P \in \mathcal{F}\}$ and morphisms $\{ \phi_{PQ} \mid Q \subseteq P\}$ where $\phi_{PQ} \colon \prod_{p \in P} X_p \rightarrow \prod_{q\in Q} X_q$ denotes the canonical projection.  The \textbf{reduced product of $X$} (with respect to $\mathcal{F}$) is the directed colimit \[\prod_{s\in S} X_s/\mathcal{F} := \varinjlim_{P \in \mathcal{F}} \left(\prod_{p\in P} X_p\right).\]  If $\mathcal{F}$ is an ultrafilter, then $\prod_{s\in S} X_s/\mathcal{F}$ is called the \textbf{ultraproduct of $X$} (with respect to $\mathcal{F}$).  If $\mathcal{F}$ is an ultrafilter and $Y^I$ is the $I$-indexed product of copies of $Y$, then $Y^I/\mathcal{F}$ is called the \textbf{ultrapower of $Y$} (with respect to $\mathcal{F}$).
\end{construction}

If $\mathcal{F}$ is a proper filter on a set $S$, then $\mathcal{F}$ is a poset with the relation $P \leq Q$ if and only if $Q \subset P$.  We consider $\mathcal{F}$ as a category and denote the morphisms corresponding to $Q \subseteq P$ by $f_{PQ} \colon P \rightarrow Q$.  We also identify the set $S$ with the discrete category with objects $S$.  Note that we have an equivalence of categories $\Der(S) \simeq \Der(\base)^S$ by (Der1).  

\begin{proposition}\label{Prop: reduced product diagram}
Let $\Der$ be a derivator and let $\mathcal{F}$ be a proper filter on a set $S$.  Then there exists a functor \[ \Redf \colon \Der(S) \rightarrow \Der(\mathcal{F})\] such that, for each diagram $X$ in $\Der(S)$, the following statements hold: \begin{enumerate}
\item The value of $\Red{X}$ at $P$ is isomorphic to $\prod_{p\in P} X_p$ for each $P \in \mathcal{F}$.
\item The value of $\Red{X}$ at $f_{PQ}$ is isomorphic to the canonical projection $\phi_{PQ}$ for each $Q \subseteq P$ in $\mathcal{F}$.
\end{enumerate}  
\end{proposition}

Since the proof of Proposition \ref{Prop: reduced product diagram} is reasonably long and technical, we will postpone it until Appendix \ref{App: Proof}.  We refer to $\Red{X}$ as the \textbf{coherent reduced product diagram} of $X$ and we define the \textbf{coherent reduced product} of $X$ to be $\Redu{X} : = \hocolim{\mathcal{F}}{\Red{X}}$.  If $\mathcal{F}$ is an ultrafilter then we will also use the terms \textbf{coherent ultraproduct} and \textbf{coherent ultrapower} where appropriate.

\begin{example}
If $\mathcal{C}$ is a bicomplete category, then the natural 2-functor $y_\mathcal{C} \colon \Cat^\op \to \CAT$ that assigns to a small category $A$ the usual diagram category $\mathcal{C}^A$ is a derivator.  In this case, the coherent reduced product conincides with the classical reduced product described in Construction \ref{Con: reduced product}.
\end{example}

\begin{example}
Let $\Der_R$ be the derivator described in Example \ref{Ex: derived category}.  Let $\mathcal{F}$ be a filter on a set $S$ and let $X$ be in $\Der_R(S)$.  Using the same notation for $X$ when considering it as an object of $\Ch{R}^S$, it follows from Example \ref{Ex: lim is lim} that $\Redu{X} \cong \prod_{s\in S} X_s /\mathcal{F}$, where the reduced product on the right is taken in $\Ch{R}$ and then considered as an object in $\Der_R(\base)$.
\end{example}

\begin{corollary}
If $\Der$ is a strong and stable derivator, then the functors $\Redf$ and $\Redu{(-)}$ are exact functors.
\end{corollary}
\begin{proof}
This is immediate from the proof of Proposition \ref{Prop: reduced product diagram}, as well as Proposition \ref{Prop: strong and stable}.
\end{proof}

\begin{remark}\label{Rem: shift red prod}
For any small category $A$, we may define the shifted derivator $\Der^A \colon \Cat^\op \rightarrow \CAT$ (see Appendix \ref{Sec: shifted}).  So, for any proper filter $\mathcal{F}$ on a set $S$ and any $X$ in $\Der(A)^S$, we define the coherent reduced product of $X$ by considering it as an object of $\Der^A(S)$.  

Moreover, for each object $a$ in $A$, the value of the coherent reduced product of $X$ at $a$ is given by the coherent reduced product of the value of $X$ at $a$.  That is,  by applying \cite[Prop.~2.5]{Der}, we have natural isomorphisms \[\xymatrix{ \Der^A(S) \ar@{}[drr]|-{\cong} \ar[rr]^{\Redf} \ar[d]_{(a\times \id_S)^*} & & \Der^A(\mathcal{F})  \ar@{}[drr]|-{\cong}\ar[rr]^{\hocolimf{\mathcal{F}}} \ar[d]^{(a\times \id_\mathcal{F})^*} & & \Der^A(\base) \ar[d]^{a^*} \\
\Der(S) \ar[rr]_{\Redf} & & \Der(\mathcal{F}) \ar[rr]_{\hocolimf{\mathcal{F}}} & & \Der(\base) 
}\] where the coherent reduced product and homotopy colimit on the top row take place with respect to $\Der^A$ and the coherent reduced product and homotopy colimit on the bottom row take place with respect to $\Der$.

For the sake of clarity, we note that this means that for each $X$ in $\Der^A(S)$ we have isomorphisms \[ (\Redf^{\Der^A}(X))_a^{\Der^\mathcal{F}} \cong \Redf^\Der(X_a^{\Der^S}) \:\: \text{ and } \: \: (\Redu{X}^{\Der^A})_a^\Der \cong \Redu{(X_a^{\Der^S})}^\Der\] where the superscripts indicate which derivator each evaluation, reduced product and homotopy colimit is taken with respect to.  Note that we consider $X$ as an object $\{X_s\}_{s\in S}$ of $\Der(A)^S$ and that $X_a^{\Der^S}$ corresponds to the object $\{(X_s)_a^\Der\}_{s\in S}$ in $\Der(\base)^S$.
\end{remark}

\section{Purity in a compactly generated derivator.}\label{Sec: Purity}

In this section we use the construction of coherent reduced products in Section \ref{Sec: Red Prod} to characterise purity in the underlying category of a compactly generated derivator.

\subsection{Purity in compactly generated triangulated categories.}
We focus on strong and stable derivators $\Der$ for which $\Der(\base)$ is a compactly generated triangulated category.  Let $\T$ be a triangulated category with arbitrary coproducts.  An object $X$ in $\T$ is called \textbf{compact} if the Hom-functor $\Hom{\T}(X,-) \colon \T \to \Ab$ commutes with arbitrary coproducts.  Then $\T$ is \textbf{compactly generated} if the full subcategory $\Tc$ of compact objects in $\T$ is skeletally small and $\Tc$ \textbf{generates} $\T$ i.e.~for every non-zero object $Y$ in $\T$ there exists some $X$ in $\Tc$ such that $\Hom{\T}(X, Y) \neq 0$.

Next we summarise some of the basic notions of purity in a compactly generated category $\T$.  Consider the category $\Mod{\Tc}$ of contravariant additive functors from $\Tc$ to the category $\Ab$ of abelian groups.  We denote the full subcategory of finitely presented functors by $\mod{\Tc}$.  

Let $\y \colon \T \to \Mod{\Tc}$ denote the \textbf{restricted Yoneda functor} which is defined to be \[\y X := \Hom{\T}(-, X)|_{\Tc} \: \:  \text{ and }\: \: \y f := \Hom{\T}(-,f)|_{\Tc}\] for objects $X$ and morphisms $f$ in $\T$.  The functor $\y$ is not always fully faithful but it enables us to consider triangles in $\T$ in terms of exact sequences in $\Mod{\Tc}$:

\begin{itemize} \item A triangle $\delta \colon X \overset{f}{\rightarrow} Y \overset{g}{\rightarrow} Z \rightarrow \Sigma X$ is called a \textbf{pure triangle} if the sequence \[ \y \delta \colon 0 \longrightarrow \y X \overset{\y f}{\longrightarrow} \y Y \overset{\y g}{\longrightarrow} \y Z \longrightarrow 0 \] is exact in $\Mod{\Tc}$.  In this case we refer to $f$ as a \textbf{pure monomorphism} and to $g$ as a \textbf{pure epimorphism}.  
\item An object $E$ is called \textbf{pure-injective} if every pure monomorphism of the form $E \rightarrow X$ is split.  An object $P$ is called \textbf{pure-projective} if every pure epimorphism of the form $X \rightarrow P$ is split.
\end{itemize}

\noindent In fact, the following proposition shows that there is a strong relationship between the injective objects in $\Mod{\Tc}$ and pure-injective objects in $\T$. 

\begin{proposition}[{\cite[Thm.~1.8]{MR1728877}}]\label{prop: pure-triangle}
Let $\T$ be a compactly generated triangulated category.  The following statements are equivalent for an object $E$ in $\T$:
\begin{enumerate}
\item $E$ is pure-injective.
\item $\y E$ is an injective object of $\Mod{\Tc}$.
\item The map $\Hom{\T}(X, E) \rightarrow \Hom{\Mod{\Tc}}(\y X, \y E)$ induced by the functor $\y$ is an isomorphism for all objects $X$ in $\T$.
\end{enumerate}
Similarly, the following statements are equivalent for an object $P$ in $\T$:
\begin{enumerate}
\item $P$ is pure-projective.
\item $\y P$ is a projective object of $\Mod{\Tc}$.
\item The map $\Hom{\T}(P, X) \rightarrow \Hom{\Mod{\Tc}}(\y P, \y X)$ induced by the functor $\y$ is an isomorphism for all objects $X$ in $\T$.
\end{enumerate}
\end{proposition}

\subsection{Pure triangles in terms of coherent reduced products.}
We will call a strong and stable derivator $\Der$ \textbf{compactly generated} if $\Der(\base)$ is a compactly generated triangulated category.  In this section we give an intrinsic characterisation of pure triangles (and hence of pure injective objects) in the underlying category $\Der(\base)$ of a compactly generated derivator $\Der$.  

\begin{lemma}\label{Lem: hcg shifts}
Let $\Der$ be a compactly generated derivator.  For any small category $A$, the category $\Der(A)$ of coherent diagrams is compactly generated.
\end{lemma}
\begin{proof}
It suffices to show that $\Der(A)$ has a set of compact generators.  We prove that the set \[\mathcal{Y} := \{Y \mid Y \cong a_!C \text{ for some $a$ in } A \text{ and $C$ in } \Der(\base)^\c\}\] is such a set.  

To see that $\mathcal{Y}$ is a generating set, let $Z$ be an object of $\Der(A)$ such that $\Hom{\Der(A)}(Y, Z) = 0$ for all $Y\in Y$.  Then we have that $\Hom{\Der(\base)}(C, Z_a) \cong \Hom{\Der(A)}(a_!C, Z) =0$ for all $a$ in $A$ and $C$ in $\Der(\base)^\c$.  Since $\Der(\base)$ is compactly generated, we have that $Z_a \cong 0$ for all $a$ in $A$.  It follows from (Der2) that $Z$ is a zero object in $\Der(A)$.  

Next we show that the objects in $\mathcal{Y}$ are compact.  For an object $a$ in $A$ and $C$ in $\Der(\base)^\c$, we have $  \Hom{\Der(\base)}(C, \left(\bigoplus_{s\in S} X_s\right)_a)  \cong \Hom{\Der(\base)}(C, \bigoplus_{s\in S} (X_s)_a) \cong  \bigoplus_{s\in S}\Hom{\Der(\base)}(C,(X_s)_a)
$ because $a^*$ is a left adjoint.  Thus $\Hom{\Der(A)}(a_!C, \bigoplus_{s\in S} X_s) \cong\bigoplus_{s\in S}\Hom{\Der(A)}(a_!C, X_s)$ as required.
\end{proof}

\begin{corollary}
Let $\Der$ be a compactly generated derivator.  For any small category $A$, the shifted derivator $\Der^A$ is compactly generated.
\end{corollary}
\begin{proof}
By Proposition \ref{Prop: strong and stable}, the derivator $\Der^A$ is strong and stable and so the statement is immediate from Lemma \ref{Lem: hcg shifts}.
\end{proof}

Let $\Der \colon \Cat^\op \rightarrow \CAT$ be a derivator.  An object $X$ in $\Der(\base)$ is called \textbf{homotopically finitely presented} if the canonical morphism \[  \varinjlim_{i\in I}\Hom{\Der(\base)}(X, Y_i) \rightarrow \Hom{\Der(\base)}(X, \hocolim{I}{Y})\] is an isomorphism for every small directed category $I$ and every object $Y$ in $\Der(I)$.  

\begin{proposition}[{\cite[Prop.~5.4]{SSV}}]
Let $\Der$ be a strong and stable derivator.  Then an object $C$ in $\Der(\base)$ is compact if and only if it is homotopically finitely presented.
\end{proposition}

\begin{remark}\label{Rem: hocolim to dir lim}
The preceding result can be rephrased in the following way: Let $\T \simeq\Der(\base)$ be the underlying category of a compactly generated derivator $\Der$.  Then, for any small directed category $I$ and any $X$ in $\Der(I)$, we have that \[\varinjlim_{i\in I} \y X_i \cong \y \hocolim{I}{X}.\]

Combining this with the construction described in Section \ref{Sec: Red Prod}, we have that any reduced product of representable functors in $\Mod{\Tc}$ is representable.  That is, for any proper filter $\mathcal{F}$ on a set $S$ and any $X$ in $\Der(S) \simeq \T^S$, we have that \[ \prod_{s\in S} \y X_s / \mathcal{F}  \cong \y \Redu{X}.\] 
\end{remark}

\begin{remark}
Reduced products and ultraproducts are ubiquitous in model theory, see for example \cite[Chap.~4]{MR1059055}. In \cite{MR2199207}, the authors introduce a language $\mathcal{L}_\T$ for a compactly generated triangulated category $\T$ such that the models of $\mathcal{L}_\T$ satisfying certain axioms coincide with the objects of $\Mod{\Tc}$. Moreover, the reduced product in the sense of \cite[Prop.~4.1.6]{MR1059055} coincides with the reduced product in $\Mod{\Tc}$.  We may consider the objects of $\T$ as models of $\mathcal{L}_\T$ via the functor $\y$.  Remark \ref{Rem: hocolim to dir lim} says that the collection of objects of $\T$ (considered as models of $\mathcal{L}_\T$) is closed under taking reduced products.
\end{remark}

\begin{proposition}\label{Prop: pure tri}
Let $\T$ be the underlying category of a compactly generated derivator $\Der$.  Let $X \rightarrow Y \rightarrow Z\rightarrow \Sigma X$ be a triangle in $\T$.  Then the following statements are equivalent:\begin{enumerate}
\item The sequence $X \rightarrow Y \rightarrow Z\rightarrow \Sigma X$ is a pure triangle.
\item There exists an ultrafilter $\mathcal{F}$ on a set $S$ such that the coherent ultrapower \[ \Redu{(X^S)} \rightarrow \Redu{(Y^S)} \rightarrow \Redu{(Z^S)} \rightarrow \Sigma\Redu{(X^S)} \] is a split triangle.
\end{enumerate}
\end{proposition}
\begin{proof}
By definition, the statement $(1)$ is equivalent to the sequence $0 \rightarrow \y X \rightarrow \y Y \rightarrow \y Z\rightarrow 0$ being an exact sequence in $\Mod{\Tc}$.  As $\y X$ is fp-injective \cite[Lem.~1.6]{MR1728877}, every exact sequence in $\Mod{\Tc}$ starting with $\y X$ is pure-exact (see Section \ref{Sec: Loc coh} for the definition of fp-injective).  By \cite[Thm.~16.1.16]{MR2530988}, this is equivalent to there existing an ultrafilter $\mathcal{F}$ on a set $S$ such that $0 \rightarrow (\y X)^S/\mathcal{F} \rightarrow (\y Y)^S/\mathcal{F} \rightarrow (\y Z)^S/\mathcal{F}\rightarrow 0$ is a split exact sequence in $\Mod{\Tc}$. Moreover, the terms in the split exact sequence can be taken to be (pure-)injective and so, by Remark \ref{Rem: hocolim to dir lim}, this is equivalent to $(2)$.
\end{proof}

\begin{remark}
The characterisation of pure-exact sequences in locally finitely presented categories in terms of ultrapowers is vital in the preceding proof.  This result follows from a standard result in model theory -- a more detailed account of this approach is given in \cite{MR2791358}.  
\end{remark}

\begin{corollary}\label{Cor: shifted pure tri}
Let $A$ be a small category.  If $X \rightarrow Y \rightarrow Z\rightarrow \Sigma X$ is a pure triangle in $\Der(A)$ then $X_a \rightarrow Y_a \rightarrow Z_a\rightarrow \Sigma X_a$ is a pure triangle in $\T$ for each $a \in A$.
\end{corollary}
\begin{proof}
Note that, by Proposition \ref{Prop: strong and stable}, the evaluation functor $a^* \colon \Der(A) \rightarrow \Der(\base)$ is an exact functor.  Combining this with Remark \ref{Rem: shift red prod}, the statement is an immediate consequence of Proposition \ref{Prop: pure tri}.
\end{proof}

\begin{remark}
A similar approach was taken in \cite{MR1899046} without requiring the presence of a derivator by using the notion of a homology colimit instead of homotopy colimit; pure triangles are characterised as homology colimits of split triangles. 
\end{remark}

\subsection{Definable subcategories in terms of coherent reduced products.}

A full subcategory $\mathcal{D}$ of a compactly generated triangulated category $\T$ is called \textbf{definable} if it is of the form \[ \mathcal{D} = \{ X \in \T \mid \Hom{\Mod{\Tc}}(F_i, \y X) = 0 \text{ for all } i \in I\}\] where $\{F_i \}_{i\in I}$ is a family of functors in $\mod{\Tc}$.  For a class of objects $\mathcal{E}$ in $\T$ we denote by $\Def{\T}{\mathcal{E}}$ the \textbf{smallest definable subcategory containing $\mathcal{E}$}.  The subcategory $\Def{\T}{\mathcal{E}}$ always exists and is given by \[\Def{\T}{\mathcal{E}} = \{ X \in \T \mid \Hom{\Mod{\Tc}}(F, \y X) = 0 \text{ for all } F \in \mathcal{Y}_\mathcal{E}\}\] where $\mathcal{Y}_\mathcal{E} = \{ F \in \mod{\Tc} \mid \Hom{\Mod{\Tc}}(F, \y M) = 0 \text{ for all } M \in \mathcal{E}\}$.  If $\mathcal{E} = \{ M \}$ then we denote the definable subcategory generated by $\mathcal{E}$ by $\Def{\T}{M}$.

The following lemma summarises the close relationship between the definable subcategories of $\T$ and the definable subcategories of $\Mod{\Tc}$ (see \cite{exactly} for the definition of definable subcategories of $\Mod{\Tc}$).  

\begin{lemma}[{\cite[Cor.~4.4]{AMV}}]\label{Lem: defT=defM}
Let $\T$ be a compactly generated triangulated category and let $\mathcal{E}$ be a class of objects in $\T$. Then $ \Def{\T}{\mathcal{E}} = \{ M \in \T \mid \y M \in \Def{\Mod{\Tc}}{\y\mathcal{E}}\}$.
\end{lemma}

\noindent Let $\mathcal{X}$ be a class of objects in $\Der(\base)$.  We say that $\mathcal{X}$ is \textbf{closed under directed homotopy colimits} if, for all directed categories $I$ and all objects $X$ in $\Der(I)$ such that $X_i \in \mathcal{X}$ for every $i$ in $I$, we have $\hocolim{I}{X}\in \mathcal{X}$.  The following theorem is a triangulated version of \cite[Cor.~4.6]{exactly}:

\begin{theorem}\label{Thm: def cats}
Let $\T$ be the underlying category of a compactly generated derivator $\Der$.  Then, for a full subcategory $\mathcal{D}$ of $\T$, the following statements are equivalent:\begin{enumerate}
\item $\mathcal{D}$ is definable;
\item $\mathcal{D}$ is closed under products, pure subobjects and directed homotopy colimits;
\item $\mathcal{D}$ is closed under pure subobjects and coherent reduced products.
\end{enumerate}
\end{theorem}
\begin{proof}
First note that if $\mathcal{D}$ is a definable subcategory, then $\mathcal{D}$ is closed under products and pure subobjects (this is immediate from \cite[Thm.~A]{MR1899046}).  

$(1) \Rightarrow (2)$:  Let $I$ be a small directed category and let $X$ be an object in $\Der(I)$ with $X_i$ contained in $\mathcal{D}$ for all objects $i$ in $I$.  By Lemma \ref{Lem: defT=defM}, we have that $\hocolim{I}{X}$ is contained in $\mathcal{D}$ if and only if $\y \hocolim{I}{X}$ is contained in $\Def{\Mod{\Tc}}{\y\mathcal{D}}$.  But the latter holds because $\y\hocolim{I}{X} \cong \varinjlim_{i\in I} \y X_i$ and $\Def{\Mod{\Tc}}{\y\mathcal{D}}$ is closed under directed colimits.

$(2) \Rightarrow (3)$: Follows immediately from the definition of coherent reduced products.

$(3) \Rightarrow (1)$:  We will show that $\Def{\T}{\mathcal{D}} \subseteq \mathcal{D}$ and hence that $\mathcal{D} = \Def{\T}{\mathcal{D}}$.  Let $X$ be in $\Def{\T}{\mathcal{D}}$ so that $\y X$ is contained in $\Def{\Mod{\Tc}}{\y\mathcal{D}}$.  By \cite[Cor.~4.10]{exactly}, there exists a proper filter $\mathcal{F}$ on a set $S$ such that there exists a pure monomorphism $\y X \rightarrow \prod_{s\in S}\y D_s / \mathcal{F}$ for some set of objects $\{ D_s\}_{s\in S}$ in $\mathcal{D}$.  In fact, by \cite[Cor.~21.3]{MR2791358} and \cite[Prop.~4.8]{exactly}, we may choose $\prod_{s\in S}\y D_s / \mathcal{F}$ to be pure-injective.  By Remark \ref{Rem: hocolim to dir lim}, we have $\prod_{s\in S}\y D_s / \mathcal{F} \cong \y \Redu{D}$ where $D$ is the object in $\Der(S)$ corresponding to $\{ D_s\}_{s\in S}$.  Note that, by \cite[Lem.~1.6]{MR1728877}, the object $\y\Redu{D}$ is absolutely pure in $\Mod{\Tc}$ (see Section \ref{Sec: Loc coh}).  Since $\y\Redu{D}$ is both absolutely pure and pure-injective, it is injective and so $\Redu{D}$ is pure-injective in $\T$.  The monomorphism $\y X \rightarrow \y\Redu{D}$ is therefore the image of a pure monomorphism $X \to \Redu{D}$ and so, by assumption, we have that $X \in \mathcal{D}$.
\end{proof}

\noindent The following is a triangulated version of \cite[Cor.~4.10]{exactly} and a generalisation of \cite[Thm.~7.5]{MR1899046}:

\begin{corollary}
Let $\mathcal{S}$ be a set of objects in $\T$.  Then $\Def{\T}{\mathcal{S}}$ consists of the collection of pure subobjects of coherent reduced products of objects in $\mathcal{S}$. 
\end{corollary}
\begin{proof}
It follows from \cite[Prop.~4.8]{exactly} (and the fact that products and directed colimits are exact in $\Mod{\Tc}$) that the collection of pure subobjects of reduced products of objects in $\mathcal{S}$ is closed under pure subobjects and coherent reduced products.
\end{proof}

\begin{corollary}\label{Cor: def subcat lift}
Let $\mathcal{V}$ be a definable subcategory of $\T$ and let $A$ be a small category.  Then \[ \mathcal{V}_A := \{ X \mid X_a \in \mathcal{V} \text{ for all objects } a \text{ in } A\}\] is a definable subcategory of $\Der(A)$.
\end{corollary}
\begin{proof}
By Remark \ref{Rem: shift red prod} and Corollary \ref{Cor: shifted pure tri}, the class $\mathcal{V}_A$ satisfies the closure conditions $(3)$ in Theorem \ref{Thm: def cats}.
\end{proof}

\section{Smashing t-structures with Grothendieck hearts.}\label{Sec: GrothHearts}

In this section we consider two kinds of smashing t-structure with Grothendieck hearts: the homotopically smashing t-structures \cite{SSV} and the pure-injective cosilting t-structures \cite{AMV}.  We follow the style of the definitions given in \cite{AMV}; in particular, a t-structure is a torsion pair in $\T$.

Let $\T$ be a triangulated category.  A \textbf{t-structure} \cite{BBD} on $\T$ is a pair ${\bf t} = (\mathcal{U}, \mathcal{V})$ of full subcategories satisfying the following conditions: \begin{description}
\item[(t1)] $\Hom{\T}(U, V) = 0$ for all $U \in \mathcal{U}$ and $V \in \mathcal{V}$;
\item[(t2)] $\Sigma \mathcal{U} \subseteq \mathcal{U}$ and $\mathcal{V} \subseteq \Sigma\mathcal{V}$;
\item[(t3)] For each object $X$ in $\T$, there exists a triangle $U \rightarrow X \rightarrow V \rightarrow \Sigma U$ where $U \in \mathcal{U}$ and $V \in \mathcal{V}$.
\end{description}

The \textbf{heart} of ${\bf t}$ is defined to be $\mathcal{G} := \Sigma^{-1}\mathcal{U}\cap\mathcal{V}$.  By \cite{MR976638}, the inclusion $\mathcal{U} \rightarrow \T$ (respectively $\mathcal{V} \rightarrow \T$) has a right adjoint $\tau_\mathcal{U} \colon \T \rightarrow \mathcal{U}$ (respectively has a left adjoint $\tau_\mathcal{V} \colon \T \rightarrow \mathcal{V}$).  These are called the \textbf{truncation functors} and the triangles in $\bf{(t3)}$ are given by: \[ \tau_\mathcal{U}(X) \rightarrow X \rightarrow \tau_\mathcal{V}(X) \rightarrow \Sigma( \tau_\mathcal{U}(X)).  \]  The associated cohomological functor to the heart $H_{{\bf t}}^0 \colon \T \rightarrow \mathcal{G}$ is defined to be \[H_{{\bf t}}^0 := \tau_\mathcal{V}\circ\Sigma^{-1}\circ\tau_\mathcal{U}\circ\Sigma = \Sigma^{-1}\circ\tau_\mathcal{U}\circ\Sigma\circ\tau_{\mathcal{V}}.\]  

We say that ${\bf t} = (\mathcal{U}, \mathcal{V})$ is \textbf{left nondegenerate} (respectively \textbf{right nondegenerate}) if $\bigcap_{i\in\mathbb{Z}} \Sigma^i\mathcal{U} = \{0\}$ (respective if $\bigcap_{i\in\mathbb{Z}} \Sigma^i\mathcal{V} = \{0\}$).  If ${\bf t} $ is both right and left nondegenerate then we say that ${\bf t} $ is \textbf{nondegenerate}.

We say that ${\bf t}$ is \textbf{smashing} if the class $\mathcal{V}$ is closed under coproducts.  For any smashing t-structure, the associated cohomological functor $H^0_{{\bf t}} \colon \T \rightarrow \mathcal{G}$ preserves coproducts in $\T$ (see \cite[Lem.~3.3]{AMV}).

\subsection{Homotopically smashing t-structures.}\label{Sec: homo smash}

The notion of a homotopically smashing t-structure was introduced in \cite{SSV}.  Let $\Der$ be a strong and stable derivator and let ${\bf t} = (\mathcal{U}, \mathcal{V})$ be a t-structure on $\Der(\base)$.  Then ${\bf t} $ is called \textbf{homotopically smashing} with respect to $\Der$ if $\mathcal{V}$ is closed under directed homotopy colimits.

\begin{remark}
Similar t-structures have been considered in the context of presentable stable $\infty$-categories.  See \cite[Rem.~1.3.5.23]{HA}.
\end{remark}

\begin{theorem}[{\cite[Thm.~A]{SSV}}] 
Let $\Der$ be a strong and stable derivator and let ${\bf t} = (\mathcal{U}, \mathcal{V})$ be a t-structure on $\Der(\base)$.  If ${\bf t} $ is homotopically smashing, then the heart $\mathcal{G}$ has exact directed colimits.
\end{theorem}

In fact, it is shown in \cite{SSV} that a homotopically smashing t-structure on the homotopy category of a combinatorial model category has a Grothendieck heart.  Next we show that if the derivator in question is compactly generated, then we do not need these additional assumptions to obtain the generators.

\begin{lemma}\label{Lemma: homo smash implies Grothendieck}
Let $\Der$ be a compactly generated derivator and let ${\bf t} = (\mathcal{U}, \mathcal{V})$ be a t-structure on $\T := \Der(\base)$. If ${\bf t} $ is homotopically smashing then the heart $\mathcal{G}$ is a Grothendieck category.
\end{lemma}
\begin{proof}
Since $\Der(\base)$ has arbitrary coproducts, it follows that $\mathcal{G}$ has arbitrary coproducts.  Also, follows from the previous theorem that $\mathcal{G}$ has exact directed colimits and so it remains to show that $\mathcal{G}$ has a set of generators.  Consider the set \[\mathcal{C} :=\{ H^0_{{\bf t}}(C) \mid C \text{ in } \Tc\}\] of objects in $\mathcal{G}$.  We will show that $\mathcal{C}$ generates $\mathcal{G}$.  Let $X$ be an object in $\mathcal{G}$ and consider $\y X$ in $\Mod{\Tc}$.  Since $\{\y C \mid C \text{ in } \Tc\}$ is a generatoring set for $\Mod{\Tc}$, there exists an epimorphism $\gamma\colon \bigoplus_{i\in I} \y C_i \rightarrow \y X \rightarrow 0$ for some set $\{C_i\}_{i\in I}$ of compact objects.  But then $\bigoplus_{i\in I} \y C_i \cong \y \left( \bigoplus_{i\in I}  C_i\right)$ and $\bigoplus_{i\in I}  C_i$ is pure-projective so $\gamma = \y g$ for some pure epimorphism $g \colon \bigoplus_{i\in I}  C_i \rightarrow X$ in $\Der(\base)$.  Since $\mathcal{G}$ has exact directed colimits and $H^0_{{\bf t}}$ preserves coproducts, we have that $H^0_{{\bf t}}$ sends pure epimorphisms to epimorphisms in $\mathcal{G}$ by \cite[Cor.~2.5]{MR1728877}.  Therefore we have that $H^0_{{\bf t}}(g) \colon \bigoplus_{i\in I} H^0_{{\bf t}}(C_i) \rightarrow X$ is an epimorphism and we have shown that $\mathcal{C}$ generates $\mathcal{G}$.
\end{proof}

\subsection{Partial cosilting t-structures.}\label{Sec: cosilt}

The cosilting t-structures can be described as particular perpendicular classes of an object. Given an object $M$ in $\T$ and a subset $I$ of $\mathbb{Z}$, we define perpendicular classes as follows: \[^{\perp_I}M := \{Y \in \T \mid \Hom{\T}(Y, \Sigma^iM) = 0 \text{ for all } i \in I\}\] and  \[M{^{\perp_I}} := \{Y \in \T \mid \Hom{\T}(M,\Sigma^iY) = 0 \text{ for all } i \in I\}.\] In what follows we will represent the set $\{ i \in \mathbb{Z} \mid i < 0\}$ by the symbol $<0$; similarly for $\leq0$, $\geq 0$ and $>0$.  Also, if $I = \{i\}$ we will simply write $\perp_i$.   This notation also applies to objects in an abelian category where Hom-spaces should be replaced by Ext-groups in the obvious way. 

An object $C$ in $\T$ is called \textbf{cosilting} if $({^{\perp_{\leq0}}}C, {^{\perp_{>0}}}C)$ defines a t-structure (which implies, in particular, that $C$ lies in ${^{\perp_{>0}}}C$).  A t-structure of the form $({^{\perp_{\leq0}}}C, {^{\perp_{>0}}}C)$ will be referred to as a \textbf{cosilting t-structure}.  We define an object $C$ in $\T$ to be \textbf{partial cosilting} if ${^{\perp_{>0}}}C$ is a coaisle and $C$ lies in ${^{\perp_{>0}}}C$; the corresponding t-structure ${\bf t} = (\mathcal{U}, {^{\perp_{>0}}}C)$ is called a \textbf{partial cosilting t-structure}.  Given a (partial) cosilting object $C$ we will denote the heart of the corresponding t-structure by $\heart{C}$.  A (partial) cosilting object $C$ is called \textbf{(partial) cotilting} if $\Prod{C} \subseteq \heart{C}$.   We say that two partial cosilting objects are \textbf{equivalent} if they give rise to the same t-structure.

In \cite[Thm.~3.6]{AMV} the authors show that the cosilting t-structures with pure-injective cosilting object parametrise the nondegenerate smashing t-structures with Grothendieck hearts.  In Theorem \ref{Thm: groth hearts}, we will make use of the following modification:

\begin{lemma}\label{Lem: weak AMV}
Let $\T$ be a compactly generated triangulated category and let ${\bf t} = (\mathcal{U}, \mathcal{V})$ be a left nondegenerate t-structure on $\T$.  If ${\bf t} $ is smashing and the heart is a Grothendieck category, then ${\bf t}$ is partial cosilting for a pure-injective partial cosilting object $C$.
\end{lemma}
\begin{proof}
Let $E$ be an injective cogenerator of $\mathcal{G}$.  By the proof of {\cite[Thm.~3.6]{AMV}}, the functor $\Hom{\mathcal{G}}(H^0_{{\bf t}}(-), E)$ is naturally isomorphic to the functor $\Hom{\T}(-,C)$ for a pure-injective object $C$.  Since ${\bf t}$ is left nondegenerate, we have that $\mathcal{V} = \{ X \in \T \mid H^0_{\bf t}(\Sigma^iX) = 0 \text{ for all } p <0\}$ and so $\mathcal{V} = {^{\perp_{>0}}C}$.  To see that $C \in {^{\perp_{>0}}C}$, let $U \in \mathcal{U}$.  Then $\Hom{\T}(U,C) \cong \Hom{\mathcal{G}}(H^0_{{\bf t}}(U), E)$.  But $H^0_{{\bf t}}(U) = 0$ and so $C \in \mathcal{U}^{\perp_0} = {^{\perp_{>0}}C}$.
\end{proof}

The next lemma follows immediately from \cite[Lem.~4.8]{AMV}; we include a different proof that makes use of the closure properties given in Theorem \ref{Thm: def cats}.

\begin{lemma}\label{Lem: pure inj oth def}
Let $\Der$ be a compactly generated derivator and let $C$ be a pure-injective object in $\mathcal{T} := \Der(\base)$.  If $\mathcal{V} := {^{\perp_{>0}}}C$ is closed under products, then $\mathcal{V}$ is a definable subcategory.
\end{lemma}
\begin{proof} 

We show that $\mathcal{V}$ has the closure properties given in Theorem \ref{Thm: def cats}.  By assumption $\mathcal{V}$ is closed under products.

Next we show that $\mathcal{V}$ is closed under pure subobjects.  So suppose $Y \in \mathcal{V}$ and consider a pure monomorphism $X \rightarrow Y$.  As $\Sigma^jC$ is pure-injective for all $j >0$, we have an induced epimorphism $ \Hom{\T}(Y, \Sigma^jC) \rightarrow \Hom{\T}(X, \Sigma^jC) \rightarrow 0$ in $\Ab$.  Then, since we have $\Hom{\T}(Y, \Sigma^jC) = 0$, it follows that $X \in {^{\perp_{>0}}}C = \mathcal{V}$.

Finally we must show that $\mathcal{V}$ is closed under directed homotopy colimits.  Suppose $I$ is a small directed category and let $X$ be an object in $\Der(I)$ with $X_i \in \mathcal{V}$ for all $i \in I$.  Then $\hocolim{I}{X} \in \mathcal{V}$ if and only if $\Hom{\T}( \hocolim{I}{X}, \Sigma^jC) = 0$ for each $j > 0$.  Note that  \begin{align*} \Hom{\T}( \hocolim{I}{X}, \Sigma^jC) & \cong \Hom{\Mod{\Tc}}(\y (\hocolim{I}{X}), \y (\Sigma^jC)) \\  &\cong \Hom{\Mod{\Tc}}(\varinjlim_{i\in I}\y (X_i), \y( \Sigma^jC)). \end{align*} 

Since we have $ 0 = \Hom{\T}(X_i, \Sigma^jC) \cong \Hom{\Mod{\Tc}}(\y( X_i), \y (\Sigma^jC))$ for all $i \in I$, we also have $\Hom{\Mod{\Tc}}(\coprod_{i\in I}\y( X_i), \y (\Sigma^jC)) =0$.  Applying the functor $\Hom{\Mod{\Tc}}(-, \y (\Sigma^jC))$ the exact sequence \[\coprod_{i\in I} \y (X_i) \rightarrow \varinjlim_{i\in I} \y (X_i) \rightarrow 0\] we conclude that \[ \Hom{\T}( \hocolim{I}{X}, \Sigma^jC) \cong\Hom{\Mod{\Tc}}(\varinjlim_{i\in I}\y (X_i), \y (\Sigma^jC)) = 0\] as desired. 
\end{proof}

\subsection{Smashing t-structures with Grothendieck hearts.}

Collecting together Lemmas \ref{Lemma: homo smash implies Grothendieck}, \ref{Lem: weak AMV}, \ref{Lem: pure inj oth def} with \cite[Thm.~3.6]{AMV}, we have proved the following theorem:

\begin{theorem}\label{Thm: groth hearts}
Let $\T$ be the underlying category of a compactly generated derivator $\Der$ and consider a (left) nondegenerate t-structure ${\bf t} = (\mathcal{U}, \mathcal{V})$ on $\T$.  Then the following statements are equivalent: \begin{enumerate}
\item ${\bf t} $ is a (partial) cosilting t-structure with a pure-injective (partial) cosilting object $C$.
\item $\mathcal{V}$ is definable.
\item ${\bf t} $ is homotopically smashing.
\item ${\bf t} $ is smashing and the heart $\mathcal{G}$ is a Grothendieck category.
\end{enumerate}
\end{theorem}

\subsection{Example via glued t-structures.}\label{Ex: recollement}
Let $\mathcal{H}$ be a locally noetherian Grothendieck category such that $\D{\mathcal{H}}$ is compactly generated.  Krause proves in \cite{MR2157133} that there exists a recollement 
\[\xymatrix@C=0.5cm{\mathrm{K_{ac}}(\Inj{\mathcal{H}}) \ar[rrr]^{I} &&& \mathrm{K}(\Inj{\mathcal{H}}) \ar[rrr]^{Q}  \ar @/_1.5pc/[lll]_{I_\lambda}  \ar @/^1.5pc/[lll]_{I_\rho} &&& \D{\mathcal{H}} \ar @/_1.5pc/[lll]_{Q_\lambda} \ar @/^1.5pc/[lll]_{Q_\rho} } \] where $\mathrm{K}(\Inj{\mathcal{H}})$ is the homotopy category of the injective objects in $\mathcal{H}$, $I$ is the inclusion of the full subcategory $\mathrm{K_{ac}}(\Inj{\mathcal{H}})$ of acyclic complexes in $\mathrm{K}(\mathcal{H})$ that are contained in $\mathrm{K}(\Inj{\mathcal{H}})$ and $Q$ is the composition of the inclusion $\mathrm{K}(\Inj{\mathcal{H}}) \rightarrow \mathrm{K}(\mathcal{H})$ with the canonical localisation $\mathrm{K}(\mathcal{H}) \rightarrow \D{\mathcal{H}}$.  

We consider a cosilting object $C$ in $\D{\mathcal{H}}$, e.g.~the injective cogenerator of $\mathcal{H}$, and show that it induces a t-structure on $\mathrm{K}(\Inj{\mathcal{H}})$ that satisfies the conditions of Theorem \ref{Thm: groth hearts}.

\begin{lemma} Let $C$ be a cosilting object in $\D{\mathcal{H}}$.  Then $\mathrm{K_{ac}}(\Inj{\mathcal{H}}) = {^{\perp_\mathbb{Z}}Q_\rho(C)}$.  
\end{lemma}
\begin{proof}
As $C$ cogenerates $\D{\mathcal{H}}$, we have that $X \in {^{\perp_\mathbb{Z}}Q_\rho(C)}$ if and only if $\Hom{\D{\mathcal{H}}}(Q(X), \Sigma^iC) = 0$ for all $i \in \mathbb{Z}$ if and only if $Q(X) = 0$ if and only if $X \in \ker{Q} = \im{I} = \mathrm{K_{ac}}(\Inj{\mathcal{H}})$.
\end{proof}

\begin{proposition}
Let $C$ be a cosilting object in $\D{\mathcal{H}}$ and consider the t-structure $(\mathcal{U}, \mathcal{V})$ obtained by gluing the trivial t-structure $(0, \mathrm{K_{ac}}(\Inj{\mathcal{H}}))$ on $\mathrm{K_{ac}}(\Inj{\mathcal{H}})$ with the cosilting t-structure $({^{\perp_{\leq0}}}C, {^{\perp_{>0}}}C)$ on $\D{\mathcal{H}}$.  Then $(\mathcal{U}, \mathcal{V})$ is a partial cosilting t-structure with partial cosilting object $Q_\rho(C)$. 
\end{proposition} 
\begin{proof}
We know that $\mathcal{V}$ is (by definition) the collection of objects $Y$ in $\mathrm{K}(\Inj{\mathcal{H}})$ such that there exists a triangle \[ X \rightarrow Y \rightarrow Z \rightarrow \Sigma X\] with $X$ acylic and $Z \cong Q_\rho(M)$ for $M \in {^{\perp_{>0}}}C$.  So consider $Y \in \mathcal{V}$. Then, by the above, we have that $X \in {^{\perp_\mathbb{Z}}Q_\rho(C)} \subseteq {^{\perp_{>0}}}Q_\rho(C)$.  Moreover, as $Q_\rho$ is fully faithful, it follows that $Z \in {^{\perp_{>0}}}Q_\rho(C)$ and so $Y \in {^{\perp_{>0}}}Q_\rho(C)$.  Conversely, for any $Y \in {^{\perp_{>0}}}Q_\rho(C)$, there is a triangle of the desired form given by the counit of $(I, I_\lambda)$ and the unit of $(Q_\rho, Q)$.  So $\mathcal{V} = {^{\perp_{>0}}}Q_\rho(C)$ is a coaisle and clearly $Q_\rho(C) \in \mathcal{V}$, so $Q_\rho(C)$ is partial cosilting.
\end{proof}

In \cite[Rem.~1]{NSZ} the authors define a ``partial cosilting object'' to be an object $X$ in $\T$ such that there exists a t-structure $({}^{\perp_{\leq 0}}X,\mathcal{V})$ and $X\in \mathcal{V}^\perp[-1]$.  If the t-structure $({}^{\perp_{\leq 0}}X,\mathcal{V})$ is left nondegenerate, then $\mathcal{V} = {}^{\perp_{>0}}X$ and the object $X$ is cosilting (see the dual of \cite[Rmk.~3]{NSZ}).  In the next proposition, we show that, in the setting of this example, the object $Q_\rho(C)$ is partial cosilting in the sense of this article and in the sense of \cite{NSZ}.  The relationship between the two notions of partial cosilting are not clear in general.

\begin{proposition}\label{Prop: NSZ partial cosilting}
Let $C$ be a cosilting object in $\D{\mathcal{H}}$ and consider the t-structure $(\mathcal{U}, \mathcal{V})$ obtained by gluing the trivial t-structure $(\mathrm{K_{ac}}(\Inj{\mathcal{H}}), 0)$ on $\mathrm{K_{ac}}(\Inj{\mathcal{H}})$ with the cosilting t-structure $({^{\perp_{\leq0}}}C, {^{\perp_{>0}}}C)$ on $\D{\mathcal{H}}$.  Then the aisle $\mathcal{U}$ coincides with ${^{\perp_{\leq0}}}Q_\rho(C)$ and $Q_\rho(C)$ is contained in $\mathcal{V}^\perp[-1]$.
\end{proposition}
\begin{proof}
The coaisle $\mathcal{V}$ is given by $Q_\rho({^{\perp_{>0}}}C)$.  We will show that $Q_\rho$ restricts to an equivalence of full subcategories $Q_\rho \colon {^{\perp_{>0}}}C \longrightarrow ({^{\perp_{\leq0}}}Q_\rho(C))^{\perp_0}$.  That is, we identify $({^{\perp_{\leq0}}}Q_\rho(C))^{\perp_0}$ with $\mathcal{V}$ and it then follows that $\mathcal{U} = {^{\perp_{\leq0}}}Q_\rho(C)$.  

Let $M \in {^{\perp_{>0}}}C$ and let $X \in {^{\perp_{\leq0}}}Q_\rho(C)$.  By using the adjuction $(Q, Q_\rho)$, we have that $Q(X) \in {^{\perp_{\leq0}}}C$ and so, using the adjunction again, we have that $Q_\rho(M) \in ({^{\perp_{\leq0}}}Q_\rho(C))^{\perp_0}$.  So $Q_\rho$ restricts to a well-defined functor and this restriction is clearly fully faithful.  It remains to show that the restricted $Q_\rho$ is dense.  As $\mathrm{K_{ac}}(\Inj{\mathcal{H}}) = {^{\perp_\mathbb{Z}}Q_\rho(C)} \subseteq {^{\perp_{\leq0}}}Q_\rho(C)$, it follows that $({^{\perp_{\leq0}}}Q_\rho(C))^{\perp_0} \subseteq \mathrm{K_{ac}}(\Inj{\mathcal{H}})^{\perp_0} = \im{Q_\rho}$.  Therefore, for $Y \in ({^{\perp_{\leq0}}}Q_\rho(C))^{\perp_0}$, we have that $Q_\rho Q(Y) \cong Y$ and so it suffices to prove that $Q(Y) \in {^{\perp_{>0}}}C$.  Let $X \in {^{\perp_{\leq0}}}C$.  Then $Q_\rho(X) \in Q_\rho({^{\perp_{\leq0}}}C) \subseteq {^{\perp_{\leq0}}}Q_\rho(C)$ and so $\Hom{\mathrm{K}(\Inj{\mathcal{H}})}(Q_\rho(X), Y) = 0$.  But then also $\Hom{\mathrm{K}(\Inj{\mathcal{H}})}(X, Q(Y)) = 0$ and $Q(Y) \in {^{\perp_{>0}}}C$ as required.

Note that the final condition holds for $Q_\rho(C)$ since the same condition holds for $C$ in $\D{\mathcal{H}}$.
\end{proof}

\section{Purity and finiteness conditions on the heart.}
Let $\mathcal{G}$ denote the heart of a t-structure ${\bf t}$.  In many interesting examples, the heart $\mathcal{G}$ has a set of finitely presented generators i.e.~$\mathcal{G}$ is \textbf{locally finitely presented}.  In this case, $\mathcal{G}$ has an internal definition of purity (see \cite{exactly} for more details).  Moreover, the purity in $\mathcal{G}$ is intimately linked with finiteness conditions on $\mathcal{G}$.  In this section we will investigate the connection between purity in $\Der(\base)$ and purity in the heart $\mathcal{G}$ of a t-structure satisfying the equivalent conditions of Theorem \ref{Thm: groth hearts}.

\subsection{Purity in the heart.}

Throughout this section let $\Der$ be a compactly generated derivator and let ${\bf t} = (\mathcal{U}, \mathcal{V})$ be a t-structure satisfying the equivalent conditions of Theorem \ref{Thm: groth hearts} with heart $\mathcal{G}$.  

\begin{lemma}\label{Lem: y and limits in the heart}
For a directed system $\{X_i\}_{i\in I}$ in the heart $\mathcal{G}$, we have \[ \y \left(\varinjlim_{i\in I} X_i\right) \cong \varinjlim_{i\in I}\y X_i.  \]
\end{lemma}
\begin{proof}
This follows directly from \cite[Thm.~A]{SSV} and Remark \ref{Rem: hocolim to dir lim}.
\end{proof}

\begin{proposition}\label{Prop: pure exact in heart is pure}
Suppose that the heart $\mathcal{G}$ is locally finitely presented.  If a short exact sequence $0 \rightarrow X \overset{f}{\rightarrow} Y \overset{g}{\rightarrow} Z \rightarrow 0$ in $\mathcal{G}$ is pure-exact, then the corresponding triangle $X \overset{f}{\rightarrow} Y \overset{g}{\rightarrow} Z \rightarrow \Sigma X$ is a pure triangle in $\Der(\base)$. 
\end{proposition}
\begin{proof}
The sequence $0 \rightarrow X \overset{f}{\rightarrow} Y \overset{g}{\rightarrow} Z \rightarrow 0$ is the directed colimit of a directed system $\{ 0 \rightarrow X_i \overset{f_i}{\rightarrow} Y_i \overset{g_i}{\rightarrow} Z_i \rightarrow 0\}_{i\in I}$ of split exact sequences in $\mathcal{G}$.  There is then a directed system of split (and hence pure) triangles $\{X_i \overset{f_i}{\rightarrow} Y_i \overset{g_i}{\rightarrow} Z_i \rightarrow \Sigma X_i\}_{i\in I}$.  Directed colimits in $\Mod{\Tc}$ are exact so \[0 \rightarrow \varinjlim_{i\in I}\y X_i \rightarrow \varinjlim_{i\in I}\y Y_i \rightarrow \varinjlim_{i\in I}\y Z_i \rightarrow 0\] is exact.  But this is $0 \rightarrow \y X \overset{\y f}{\rightarrow} \y Y \overset{\y g}{\rightarrow} \y Z \rightarrow 0$ by Lemma \ref{Lem: y and limits in the heart} and so $X \overset{f}{\rightarrow} Y \overset{g}{\rightarrow} Z \rightarrow \Sigma X$ is pure.
\end{proof}

\begin{remark} The product (in $\mathcal{G}$) of a set $\{X_s\}_{s\in S}$ of objects in $\mathcal{G}$ is given by $H^0_{{\bf t}}(\prod_{s\in S} X_s)$ where the symbol $\prod$ denotes the product in $\Der(\base)$.  It follows that, in general, the definable subcategories of $\mathcal{G}$ may not be definable subcategories of $\Der(\base)$. \end{remark}

\begin{corollary}\label{Cor: pure exact in heart is pure}
Let $\mathcal{D} \subseteq \mathcal{G}$ be a definable subcategory of $\T \simeq \Der(\base)$.  If $\mathcal{G}$ is locally finitely presented, then $\mathcal{D}$ is a definable subcategory of $\mathcal{G}$.
\end{corollary}
\begin{proof}
We use the closure conditions $(2)$ of Theorem \ref{Thm: def cats} to show that $\mathcal{D}$ is closed under pure subobjects, products and directed colimits in $\mathcal{G}$.  It follows immediately from Proposition \ref{Prop: pure exact in heart is pure} that $\mathcal{D}$ is closed under pure subobjects in $\mathcal{G}$.  For closure under products in $\mathcal{G}$, let $\{X_s\}_{s\in S}$ be a set of objects in $\mathcal{D}$.  Then $\prod_{s\in S} X_s \in \mathcal{D} \subseteq \mathcal{G}$ and so the product taken in $\mathcal{G}$ is $H^0_{{\bf t}}(\prod_{s\in S} X_s) \cong \prod_{s\in S} X_s$ and hence is contained in $\mathcal{D}$ as required.  Finally let $\{X_i\}_{i\in I}$ be a directed system in $\mathcal{D}$.  Then $\y \varinjlim_{i\in I} X_i \cong \varinjlim_{i\in I}\y X_i$ is contained in $\Def{\Mod{\Tc}}{\y\mathcal{D}}$.  Thus we have that $\varinjlim_{i\in I}X_i \in \mathcal{D}$ by Lemma \ref{Lem: defT=defM}.
\end{proof}

\subsection{Locally noetherian hearts.}\label{Sec: loc noeth}
In this section, we work in a compactly generated triangulated category $\T$ that we do not require to be the underlying category of a derivator.  An object in a Grothendieck category $\mathcal{H}$ is called \textbf{noetherian} if the set of its subobjects satisfies the ascending chain condition.  The category $\mathcal{H}$ is called \textbf{locally noetherian} if there is a set of noetherian generators. 

We will need the properties of a partial cosilting object $C$ stated in the next lemma.  A proof for each statement essentially already exists in the literature in some form but we adapt the statements to fit in the current setting.  

\begin{lemma}\label{lem: injectives}
Let $C$ be a partial cosilting object in $\T$ and let ${\bf t} = (\mathcal{U}, \mathcal{V})$ denote the associated t-structure.  Then the following statements hold. \begin{enumerate}
\item $\Prod{C} = \mathcal{V} \cap (\Sigma^{-1}\mathcal{V})^{\perp_0}$.
\item The cohomological functor $H_{{\bf t}}^0$ restricts to an equivalence \[H_{{\bf t}}^0 \colon \Prod{C} \overset{\sim}{\rightarrow} \Prod{H_{{\bf t}}^0(C)}.\]
\item $H_{{\bf t}}^0(C)$ is an injective cogenerator for $\heart{C}$.
\end{enumerate}
\end{lemma}
\begin{proof} By definition $\mathcal{V} = {^{\perp_{>0}}}C$.
(1) This follows by the same argument as in \cite[Lem.~4.5]{PV}; (2) This follows by the same argument as in \cite[Lem.~2.8]{AMV}; (3) The fact that the image $H^0_{{\bf t}}(\Prod{C})$ consists of injective objects is the dual of \cite[Lem.~2(1)]{NSZ}.  To show that $H^0_{{\bf t}}(C)$ is a cogenerator, let $M$ be an object contained in $\heart{C}$ such that $\Hom{\T}(M, C)=0$.  Then, in fact, $M \in {^{\perp_{\geq0}}}C = \Sigma^{-1}\mathcal{V}$.  So $M \in \Sigma^{-1}\mathcal{U}\cap \Sigma^{-1}\mathcal{V} = 0$.  So for all non-zero $M$ in $\heart{C}$, we have $\Hom{\T}(M,C) \neq0$.  As $C \in \mathcal{V}$ we may consider the following triangle obtained by shifting and rotating the truncation triangle for $\Sigma C$: $\Sigma^{-2}(\tau_\mathcal{V}(\Sigma C)) \rightarrow H^0_{{\bf t}}(C) \rightarrow C \rightarrow \Sigma^{-1}(\tau_\mathcal{V}(\Sigma C)).$  Applying $\Hom{\T}(M,-)$ and using that $M \in \Sigma^{-1}\mathcal{U} = {^{\perp_0}(\Sigma^{-1}\mathcal{V})}$ we obtain that $\Hom{\T}(M,C) \cong \Hom{\heart{C}}(M, H^0_{{\bf t}}(C))$ and so $H^0_{{\bf t}}(C)$ is indeed an injective cogenerator for $\heart{C}$.
\end{proof}

We can characterise locally noetherian Grothendieck categories in terms of properties of the injective cogenerator and we use this in the next proposition to identify the cosilting objects whose associated heart is locally noetherian.  

An injective object $E$ is \textbf{$\Sigma$-injective} if $E^{(I)}$ (i.e.~the $I$-indexed direct sum of copies of $E$) is injective for every set $I$.  A Grothendieck category $\mathcal{H}$ is locally noetherian if and only if $\mathcal{H}$ has a $\Sigma$-injective cogenerator (see, for example, \cite[Cor.~3]{MR0367017}).  Similarly, we may define a pure-injective object $N$ to be \textbf{$\Sigma$-pure-injective} if $N^{(I)}$ is pure-injective for every set $I$.

\begin{proposition}\label{Prop: loc noeth}
Let $C$ be a pure-injective partial cosilting object.  If $C$ is $\Sigma$-pure-injective, then heart $\heart{C}$ is locally noetherian.  The converse statement holds when $C$ is cotilting.
\end{proposition}
\begin{proof}
Suppose that $C$ is a $\Sigma$-pure-injective partial cosilting object.  For any set $I$, the pure monomorphism $C^{(I)} \to C^I$ splits and so $C^{(I)}$ is contained in $\Prod{C}$. Since the cosilting t-structure is smashing, we have that $H_{{\bf t}}^0(C)^{(I)} \cong H_{\bf t}^0(C^{(I)})$ is contained in $\Prod{H_{{\bf t}}^0(C)}$ and so $H_{\bf t}^0(C)^{(I)}$ is injective by Lemma \ref{lem: injectives}.  As $H_{\bf t}^0(C)$ is an injective cogenerator of $\heart{C}$, we have shown that $\heart{C}$ is locally noetherian.

Now assume that $\heart{C}$ is locally noetherian and that $C$ is cotilting, so that $H_{\bf t}^0(C) \cong C$.  The assumption on $\heart{C}$ tells us that $H_{\bf t}^0(C)^{(I)} \cong C^{(I)}$ is an injective object of $\heart{C}$ or, equivalently, is contained in $\Prod{H_{\bf t}^0(C)} = \Prod{C}$ and hence is pure-injective.  
\end{proof}

\begin{example}\label{ex: counter-example} The following example was communicated to the author by Michal Hrbek.  Let $\T = \D{\Mod{\Z}}$ and consider the complex $C = \Q \oplus J_\P[-1]$ where $J_\P := \prod_{p\in \P} J_p$ is the product of the $p$-adic integers indexed by the set $\P$ of all prime numbers.  By \cite[Ex.~6.2]{AH}, the object $C$ is a cosilting object of $\T$ corresponding to the filtration \[\dots = \mathrm{Spec}(\Z) = \mathrm{Spec}(\Z) \supset \mathrm{Max}(\Z) \supset \mathrm{Max}(\Z) \supset \emptyset = \emptyset = \dots\] of $\mathrm{Spec}(\Z)$ by specialisation closed subsets of $\mathrm{Spec}(\Z)$. It can be shown that $\heart{C}$ is equivalent to $\Mod{\Q}\times\mathcal{X}$ where $\mathcal{X} \subseteq \Mod{\Z}$ is the full subcategory of torsion abelian groups.  See also \cite[Ex.~6.19]{PV}.  The heart $\heart{C}$ is locally noetherian but $J_p$ is not $\Sigma$-pure-injective for any $p\in \P$ and so the complex $C$ is not $\Sigma$-pure-injective.  This shows that it is necessary to assume that $C$ is a cotilting object for the converse statement to hold in general.
\end{example}

\subsection{Locally coherent hearts.}\label{Sec: Loc coh}
Let ${\bf t} = (\mathcal{U},\mathcal{V})$ be a partial cosilting t-structure on a compactly generated triangulated category $\T$ with pure-injective partial cosilting object $C$ and heart $\heart{C}$.  In this section we address the question of when $\heart{C}$ is locally coherent.  For a Grothendieck category $\mathcal{H}$, let $\fp{\mathcal{H}}$ denote the full subcategory of finitely presented objects in $\mathcal{H}$.  The category $\mathcal{H}$ is called \textbf{locally coherent} if $\mathcal{H}$ is locally finitely presented and $\fp{\mathcal{H}}$ is an abelian exact subcategory of $\mathcal{H}$.

In the same vein as Section \ref{Sec: loc noeth}, we consider which properties of an injective cogenerator $E$ distinguish a locally coherent Grothedieck category from a general Grothedieck category.  In order to do this, we consider the class $\Fpinj{\mathcal{H}}$ of fp-injective objects.  An object $X$ in $\mathcal{H}$ is called \textbf{fp-injective} if $\Ext{\mathcal{H}}(F, X) = 0$ for all objects $F$ in $\fp{\mathcal{H}}$.  An object $X$ is fp-injective if and only if $X$ is \textbf{absolutely pure} i.e.~every exact sequence of the form $0 \rightarrow X \rightarrow Y \rightarrow Z \rightarrow 0$ is pure-exact.  It follows directly from this characterisation, that the class of fp-injective objects coincides with $\ele{E}$ i.e.~the class of pure subobjects of products of copies of $E$.  

Suppose $\mathcal{H}$ is a locally finitely presented category.  We say that a pure-injective object $N$ is an \textbf{elementary cogenerator} if $\ele{N}$ coincides with the smallest definable subcategory of $\mathcal{H}$ containing $N$.  

\begin{remark} Elementary cogenerators were first considered in the context of model theory of modules (see, for example, \cite[Sec.~9.4]{MR933092}).  It turns out that, if $\mathcal{H}$ is a module category, then every definable subcategory $\mathcal{D}$ of $\mathcal{H}$ ``has" an elementary cogenerator in the sense that $\mathcal{D} = \ele{N}$ for some pure injective object $N$ in $\mathcal{D}$.  This follows from the fact that the elementary cogenerators are closely related to the injective cogenerators of certain locally coherent Grothendieck categories (the analogous results for compactly generated triangulated categories are contained in Proposition \ref{prop: ele cog = inj cog}). \end{remark}

We then have the following characterisation of locally coherent Grothedieck categories.

\begin{proposition}\label{prop: loc coh fp inj}
Let $\mathcal{H}$ be a Grothedieck category with injective cogenerator $E$.  The following statements are equivalent. \begin{enumerate}
\item The category $\mathcal{H}$ is locally coherent.
\item The class $\Fpinj{\mathcal{H}}$ is closed under direct limits.
\item The category $\mathcal{H}$ is locally finitely presented and $E$ is an elementary cogenerator.
\end{enumerate}
\end{proposition}
\begin{proof}
By combining (the proof of) \cite[Thm.~3.2]{MR0258888} with \cite[Prop.~3.5]{MR3613439} we have the equivalence between (1) and (2).  For the final statement, note that the class $\Fpinj{\mathcal{H}} = \ele{E}$ is always closed under pure-subobjects and products and so, by \cite[Cor.~4.6]{exactly}, we have that $\ele{E}$ is definable if and only if it is closed under direct limits.
\end{proof}

We may adapt the notion of elementary cogenerator to pure-injective objects in $\T$.  That is, a pure-injective object $X$ in $\mathcal{T}$ is called an \textbf{elementary cogenerator} if $\Def{\T}{X} = \ele{X}$ where $\ele{X}$ denotes the class of all pure subobjects of products of copies of $X$ in $\T$.  

Next we show that the elementary cogenerators in $\T$ are closely related to the injective cogenerators of the quotients of $\Mod{\Tc}$ by hereditary torsion classes of finite type.  A torsion pair $(\mathcal{A}, \mathcal{B})$ in a Grothendieck category $\mathcal{H}$ is said to be of \textbf{finite type} if $\mathcal{B}$ is closed under directed colimits.  It is well-known that the hereditary torsion classes $\mathcal{A}$ of $\mathcal{H}$ are exactly the localising subcategories of $\mathcal{H}$, and so we may form the quotient category $\mathcal{H}/\mathcal{A}$.  If the hereditary torsion pair $(\mathcal{A}, \mathcal{B})$ is of finite type, then both $\mathcal{A}$ and $\mathcal{H}/\mathcal{A}$ are locally coherent ({\cite[Thm.~2.6]{MR1426488}, \cite[Thm.~2.16]{MR1434441}}).

The definable subcategories of $\T$ parametrise the hereditary torsion pairs of finite type in $\Mod{\Tc}$ in the following way.

\begin{proposition}[\cite{MR1899046}]\label{prop: def-torsion}
Let $\T$ be a compactly generated triangulated category.  Then there is a bijective correspondence between the following sets: \begin{itemize}
\item The set of definable subcategories $\mathcal{D}$ of $\T$.
\item The set of hereditary torsion pairs $(\mathcal{A}, \mathcal{B})$ of finite type in $\Mod{\Tc}$.
\end{itemize} with mutually inverse bijections given as follows:
\[ \mathcal{D} \mapsto (\mathcal{A}_\mathcal{D}, \mathcal{B}_\mathcal{D}) \text{ where } \mathcal{A}_\mathcal{D} = \{ F\in \Mod{\Tc} \mid \Hom{\Mod{\Tc}}(F, \y D) =0 \text{ for all } D \in \Pinj{\mathcal{D}} \} \]
\[ \text{ and } (\mathcal{A}, \mathcal{B}) \mapsto \mathcal{D}_{(\mathcal{A},\mathcal{B})} = \{ D \in \mathcal{T} \mid \Hom{\Mod{\Tc}}(F, \y D) = 0 \text{ for all } F \in \mathcal{A}\cap\mod{\Tc}\}.\]
\end{proposition}

Therefore, for every definable subcategory $\mathcal{D}$ of $\T$, we obtain a quotient category and a canonical localisation functor $Q_\mathcal{D} \colon \Mod{\Tc} \to \Mod{\Tc}/\mathcal{A}_\mathcal{D}$.  From this point of view, we may make the first connection between elementary cogenerators in $\T$ and injective cogenerators.

\begin{proposition}\label{prop: ele cog = inj cog}
Let $\T$ be a compactly generated triangulated category and let $X$ be a pure-injective object in a definable subcategory $\mathcal{D}$ of $\T$.  The following statements are equivalent. \begin{enumerate}
\item $X$ is an elementary cogenerator and $\mathcal{D} = \Def{\T}{X}$.
\item $Q_\mathcal{D}(\y X)$ is an injective cogenerator of $\Mod{\Tc}/\mathcal{A}_\mathcal{D}$.
\end{enumerate}
Moreover, for every definable subcategory $\mathcal{D}$ of $\T$ there exists an elementary cogenerator $X$ such that $\mathcal{D} = \ele{X}$.
\end{proposition}
\begin{proof}
Note that, since the injective objects of $\Mod{\Tc}/\mathcal{A}_\mathcal{D}$ are given by the injective objects in $\mathcal{B}_\mathcal{D}$, it follows from Proposition \ref{prop: def-torsion}, we have that \[\Inj{\Mod{\Tc}/\mathcal{A}_\mathcal{D}} = \{Q_\mathcal{D}(\y Y) \mid Y\in\mathcal{D} \text{ is pure-injective }\}.\]  Therefore, it follows that $Q_\mathcal{D}(\y X)$ is an injective cogenerator of $\Mod{\Tc}/\mathcal{A}_\mathcal{D}$ if and only if $\Inj{\Mod{\Tc}/\mathcal{A}_\mathcal{D}} = \Prod{Q_\mathcal{D}(\y X)}$ if and only if $\mathcal{B}_\mathcal{D} = \Cogen{\y X}$ if and only if $\mathcal{D} = \ele{X}$ i.e.~$X$ is an elementary cogenerator.  The final statement follows immediately since, for every definable subcategory $\mathcal{D}$, the quotient category $\Mod{\Tc}/\mathcal{A}_\mathcal{D}$ has an injective cogenerator.
\end{proof}

Returning to the question of when $\heart{C}$ is locally coherent for a pure-injective partial cosilting object $C$, we refer to \cite{AMV} where the authors show that $\heart{C}$ is equivalent to a localisation of $\Mod{\Tc}$.  Given Lemma \ref{lem: injectives}, we observe that their proof of \cite[Thm.~3.6]{AMV} extends to partial cosilting objects.  That is, we have an equivalence of categories
\[ \heart{C} \simeq \Mod{\Tc}/ {^{\perp_0}\y C}
\] where ${^{\perp_0}\y C}$ is a hereditary torsion class in $\Mod{\Tc}$ and the torsion-free class is given by $\Cogen{\y C} := \{ F \in \Mod{\Tc} \mid  F \hookrightarrow \y C^S \text{ for some set } S \}$.  Given this description of $\heart{C}$, the following result is almost immediate.

\begin{proposition}\label{Prop: ele cogen loc coh}
Let $\T$ be a compactly generated triangulated category and let $(\mathcal{U},\mathcal{V})$ be a partial cosilting t-structure with heart $\heart{C}$.  If $C$ is an elementary cogenerator, then $\heart{C}$ is locally coherent.
\end{proposition}
\begin{proof}
If $C$ is an elementary cogenerator, then the hereditary torsion pair $({^{\perp_0}\y C}, \Cogen{C})$ is the image of $\Def{\T}{C}$ under the correspondence in Proposition \ref{prop: def-torsion}.  In particular, the torsion pair $({^{\perp_0}\y C}, \Cogen{C})$ is of finite-type and so $\heart{C}$ is locally coherent by {\cite[Thm.~2.6]{MR1426488} or \cite[Thm.~2.16]{MR1434441}}.
\end{proof}

As our final result, we show that, in the case where $C$ is a partial cotilting object of $\Der(\base)$ for a compactly generated derivator $\Der$, the converse of Proposition \ref{Prop: ele cogen loc coh} also holds.  Throughout the next lemma, the projection morphism from a product $X^H$ to the component indexed by $h\in H$ will be denoted by $\pi_h$.

\begin{lemma}\label{lem: direct system}
Let $\T$ be a compactly generated triangulated category and let $Y$ be a pure-injective object in $\T$.  If $\{X_i\}_{i\in I}$ is a directed system in $\ele{Y}$, then there is a directed system $\{f_i \colon X_i \rightarrow Y^{S_i}\}_{i\in I}$ of pure monomorphisms consisting of the universal morphisms $f_i \colon X_i \rightarrow Y^{S_i}$ where $S_i := \Hom{\T}(X_i,Y)$ and $\pi_h \circ f_i = h$ for each $h \in S_i$.
\end{lemma}
\begin{proof}
For every $X\in \ele{Y}$, any $\Prod{Y}$-preenvelope is a pure monomorphism.  This applies, in particular, to the universal morphisms $f_i \colon X_i\rightarrow Y^{S_i}$.

Let $i < j$ in $I$ and let $a_{ij} \colon X_i \rightarrow X_j$ be the corresponding morphism in the directed system.  Then, by the universal property of the product $Y^{S_j}$, there is a unique morphism $b_{ij} \colon Y^{S_i} \rightarrow Y^{S_j}$ such that, for each $h\in S_j$, we have $\pi_h \circ b_{ij} = \pi_k$ where $k := \pi_h \circ f_j \circ a_{ij} \in S_i$.  Note that $\pi_h \circ f_j \circ a_{ij} = \pi_h \circ b_{ij} \circ f_i$ for all $h\in S_j$ and so, again by the universal property of the product $Y^{S_j}$, we have that the square \[ \xymatrix{X_i \ar[r]^{f_i} \ar[d]_{a_{ij}} & Y^{S_i} \ar[d]^{b_{ij}} \\ X_j \ar[r]_{f_j} & Y^{S_j}}\] commutes.  Finally we show that the collection of such squares define a directed system.  Consider $i<j<k$ in $I$.  We must show that $b_{jk} \circ b_{ij} = b_{jk}$.  By definition, the morphism $b_{ik}$ is the unique morphism such that $\pi_r \circ b_{ik} = \pi_q$ for all $r\in S_k$ where $q := \pi_r \circ f_k \circ a_{ik}$.  But $b_{jk} \circ b_{ij}$ satisfies this property, so $b_{jk} \circ b_{ij} = b_{jk}$ as desired.
\end{proof}

\begin{theorem}\label{Thm: loc coh ele cogen}
Let $\T$ be the underlying category of a compactly generated derivator $\Der$ and consider a partial cotilting t-structure ${\bf t} = (\mathcal{U},\mathcal{V})$ on $\T$ with pure-injective partial cotilting object $C$.  Then $\heart{C}$ is locally coherent if and only if $C$ is an elementary cogenerator.
\end{theorem}
\begin{proof}
One direction is Proposition \ref{Prop: ele cogen loc coh}, we prove the converse.  The composition and product of pure monomorphisms is a pure monomorphism, so it is clear that $\ele{C}$ is closed under pure subobjects and products.  By Theorem \ref{Thm: def cats}, it remains to show that $\ele{C}$ is closed under directed homotopy colimits.  Let $I$ be a small directed category and let $X \in \Der(I)$ with $X_i \in \ele{C}$.  Then, as in  Lemma \ref{lem: direct system}, there exists a pure monomorphism $X_i \overset{f_i}{\rightarrow} C^{S_i}$ where $S_i := \Hom{\T}(X_i, C)$ for each $i\in I$.  Denote by $a_{ij} \colon X_i \rightarrow X_j$ the morphisms in $\dia{I}X$.  This induces a directed system $\{\y X_i\}_{i\in I}$ in $\Mod{\Tc}$. Since $C$ is pure-injective, we apply Lemma \ref{lem: direct system} to obtain a directed system of monomorphisms \[\xymatrix{ 0 \ar[r] & \y X_i \ar[d]_{\y a_{ij}} \ar[r]^{\y f_i} & \y C^{S_i} \ar[d]^{\y b_{ij}} \\ 0 \ar[r] & \y X_j \ar[r]_{\y f_j} & \y C^{S_j} }\]  for some $C^{S_i} \overset{b_{ij}}{\rightarrow} C^{S_j}$ in $\Der(\base)$.

Then the morphism $\varinjlim_{i\in I} \y f_i \colon \varinjlim _{i\in I}\y X_i \longrightarrow \varinjlim_{i \in I} \y C^{S_i}$ is a monomorphism.  Since $C$ is partial cotilting, it is contained in $\heart{C}$ and so we have that \[\y \left(\varinjlim_{i\in I} C^{S_i}\right) \cong \varinjlim_{i\in I}\y C^{S_i}\] by Lemma \ref{Lem: y and limits in the heart}.  Moreover, since $\heart{C}$ is locally coherent we may apply Proposition \ref{prop: loc coh fp inj}, so the object $\varinjlim_{i\in I} C^{S_i}$ is fp-injective. Since $C$ is an injective cogenerator of $\heart{C}$, there exists a pure monomorphism $\varinjlim_{i\in I} C^{S_i} \overset{h}{\rightarrow} C^J$ for some set $J$.  Composing $\y h$ with $\varinjlim_{i\in I} \y f_i$ we obtain a monomorphism $\y \hocolim{I}{X} \cong \varinjlim_{i\in I} \y X_i \rightarrow \y C^J$ in $\Mod{\Tc}$.  As $C^J$ is pure-injective, this is induced by a pure monomorphism $\hocolim{I}{X}  \rightarrow C^J$.  That is $\hocolim{I}{X}\in \ele{C}$.
\end{proof}

\begin{example}\label{Ex: ele modules are ele}
Consider the compactly generated derivator $\Der_R$ from Example \ref{Ex: derived category}.  We may consider elementary cogenerators in the heart $\mathcal{G} \simeq \Mod{R}$ of the standard t-structure in $\Der_R(\base) \simeq \DModR$. As this is a definable subcategory of $\Der_R(\base)$ (defined by the functors $\Hom{\Der_R(\base)}(\Sigma^iR, -)$ for $i \neq 0$), a module in is an elementary cogenerator in $\Mod{R}$ if and only if it is an elementary cogenerator in $\DModR$.  In particular, let $C$ be a cotilting module in the sense of \cite{AHC}.  Then $C$ is an elementary cogenerator if and only if the heart $\heart{C}$ of the associated t-structure (see \cite[Sec.~4]{MR3239134}) is locally coherent.
\end{example}

\begin{example} Consider the cosilting object $C = \Q \oplus J_\P[-1]$ of $\T = \D{\Mod{\Z}}$ given in Example \ref{ex: counter-example}.  We have already seen that the heart $\heart{C}$ is locally noetherian and so it is locally coherent.  However $C$ is not an elementary cogenerator. Indeed, by \cite[Thm.~8.1]{MR2199207}, the definable closure $\Def{\T}{C}$ of $C$ in $\T$ contains the definable closure of $J_\P[-1]$ in $\Mod{R}[-1]$.  In particular, the object $\Q[-1]$ is an object of $\Def{\T}{C}$ (see, for example, \cite[Sec.~5.2.1]{MR2530988}) but is not contained in $\Prod{C}$.  This example shows that the assumption that $C$ is a cotilting object in Theorem \ref{Thm: loc coh ele cogen} is a necessary one.
\end{example}

\appendix 
\section{Appendix: The axioms (Der1)-(Der4) and shifted derivators.}\label{App: axioms}
\subsection{The axioms.}\label{Sec: axioms}
We will now state the axioms defining a derivator.  In order to state (Der4) we will need the following definition.  Let $u \colon A \rightarrow B$ be a morphism in $\Cat$ and $b$ be an object in $B$.  Then we may form the \textbf{comma category} $u/b$ as follows: the objects of $u/b$ are given by pairs $(a,f)$ with $a$ an object in $A$ and $f \colon u(a) \rightarrow b$. The morphisms $(a,f) \rightarrow (a',f')$ in $u/b$ are given by morphisms $g \colon a\rightarrow a'$ in $A$ such that $f = f' \circ u(g) $.  Let $p \colon u/b \rightarrow A$ be the obvious projection functor.  We may perform the dual construction to obtain the comma category $b/u$ and projection functor $q \colon b/u \rightarrow A$.

A prederivator $\Der$ is a \textbf{derivator} if it has the following properties. \begin{description}
\item[(Der1)] For every small family $\{A_i\}_{i\in I}$ of small categories, the canonical functor \[\Der(\coprod_{i\in I} A_i) \rightarrow \prod_{i\in I} \Der(A_i)\] is an equivalence of categories.
\item[(Der2)] For every small category $A$, a morphisms $f \colon X \rightarrow Y$ in $\Der(A)$ is an isomorphism if and only if $f_a \colon X_a \rightarrow Y_a$ is an isomorphism for every object $a$ in $A$.
\item[(Der3)] For all functors $u \colon A \rightarrow B$, the restriction functor $u^* \colon \Der(B) \rightarrow \Der(A)$ has a left adjoint $u_! \colon \Der(A) \rightarrow \Der(A)$ and a right adjoint $u_* \colon \Der(A) \rightarrow \Der(B)$.
\item[(Der4)] For all functors $u \colon A \rightarrow B$ and all objects $b$ in $B$, there are canonical isomorphisms $\pi_!p^* \rightarrow b^*u_!$ and $b^*u_* \rightarrow \pi_*q^*$.
\end{description}

The functor $\pi_*$ is a homotopy limit functor (see the next subsection) and so (Der4) means that, for all functors $u \colon A \rightarrow B$, the image of the right Kan extension $u^*$ can be expressed point-wise in terms of these simpler right Kan extensions.  A similar statement may be made for left Kan extensions.

The canonical isomorphisms arising in (Der4) are instances of canonical mate transformations.  Many of the proofs in the later sections of this paper will refer to the calculus of canonical mates and the existence of homotopy exact squares.  For a systematic treatment of these techniques, we refer the reader to \cite[Sec.~1.2]{Der}.

\subsection{Shifted derivators.}\label{Sec: shifted}

Let $B$ be a small category and consider the 2-functor $B \times - \colon \Cat^\op \rightarrow \Cat^\op$ taking each $A$ to the product $B \times A$.  Then the \textbf{shifted derivator} $\Der^B$ is defined to be the derivator $\Der$ precomposed with $B\times -$.  This is clearly a 2-functor and in \cite[Thm.~1.25]{Der} it is shown that $\Der^B$ is a derivator.

The following definitions describe the restriction functors and Kan extensions in the shifted derivator.  We have added decorations to indicate which derivator they have been taken with respect to.  We will also use this notation in later sections when necessary: \begin{itemize}
\item For each small category $A$, we have that $\Der^B(A) := \Der(B\times A)$;
\item For each functor $u \colon A \rightarrow C$ in $\Cat$, we have that $u_{\Der^B}^* := (\id_B \times u )^*_\Der$, $u^{\Der^B}_* := (\id_B \times u )_*^\Der$ and $u^{\Der^B}_! := (\id_B \times u )_!^\Der$; 
\item The evaluation functors and the functors $\hocolimf{A}^{\Der^B}$, $\holimf{A}^{\Der^B}$ and $\dia{A}^{\Der^B}$ are all defined as in Sections \ref{Sec: dia} and \ref{Sec: holim} using the above definitions.
\item By \cite[Prop.~2.5]{Der} we have that $\hocolimf{A}^\Der(X_b^{\Der^A}) \cong \hocolimf{A}^{\Der^B}(X)_b^\Der$ and $\holimf{A}^\Der(X_b^{\Der^A}) \cong \holimf{A}^{\Der^B}(X)_b^\Der$ for all $X$ in $\Der(B\times A)$ and $b$ in $B$.
\end{itemize}

\begin{proposition}[{\cite[Prop.~4.3]{Der}}]
Let $\Der$ be a strong and stable derivator.  For any small category $A$, the shifted derivator $\Der^A$ is strong and stable.
\end{proposition}

\begin{example}
Let $\kk$ be a field and let $Q$ be a finite quiver.  Then we can consider the free category generated by $Q$ and so we can also consider $\Der_\kk(Q)$.  Unravelling the definitions, we have that $\Der_\kk(Q)$ is equivalent to the derived category $\D{\Mod{\kk Q}}$ of modules over the path algebra $\kk Q$.  Now, for every small category $A$, we have that $(\Mod{\kk})^{Q\times A} \cong (\Mod{\kk Q})^A$ and so $\Der_{\kk Q}(A) \cong \Der_\kk(Q\times A)$.  The derivator $\Der_{\kk Q}$ is therefore the shifted derivator $\Der_\kk^Q$.
\end{example}

\section{Appendix: Proof of Proposition \ref{Prop: reduced product diagram}.}\label{App: Proof}

For the proof of Proposition \ref{Prop: reduced product diagram}, we will require the following lemma, which was shared with the author by Moritz Groth.  For a small category $A$, let $A^\triangleleft$ denote the category obtained from $A$ by adding a new initial object $-\infty$ and let $i_A \colon A \rightarrow A^\triangleleft$ be the canonical inclusion.  As in \cite{Groth}, we will call an object $X$ in $\Der(A^\triangleleft)$ a \textbf{limiting cone} if it is in the essential image of $(i_A)_*$.

\begin{lemma}\label{Lem: Moritz}
Let $\Der$ be a derivator and let $S$ be a discrete category.  An object $X$ in $\Der(S^\triangleleft)$ is a limiting cone if and only if the underlying diagram $\dia{S^\triangleleft}(X)$ is a product cone i.e.~$\dia{S^\triangleleft}(X)$ exhibits $X_{-\infty}$ as the product of the objects $\{X_s\}_{s\in S}$ in $\Der(\base)$.
\end{lemma}
\begin{proof}
Consider the diagram 
\[\xymatrix{ \Der(\base)^{S^\triangleleft} \ar[rr]^{i_S^*} & & \Der(\base)^S \dlltwocell\omit{\id} \ar@/^1pc/[dr]^{\holimf{S}} & \dltwocell\omit{\cong} \\
\Der(S^{\triangleleft}) \ar[rr]_{i_S^*} \ar[u]^{\dia{S^\triangleleft}} & & \Der(S) \dlltwocell\omit{\alpha^*} \ar[u]^{\dia{S}} \ar[r]^{\pi_*} & \Der(\base) \dltwocell\omit{\eta} \\
\Der(S^{\triangleleft}) \ar[rr]_{\infty^*} \ar[u]^{\id} & & \Der(\base) \ar[u]_{\pi^*} \ar@/_1pc/[ur]_{\id}  &
}\] where the top right triangle is a natural isomorphism by \cite[Prop.~1.7]{Der}; the natural transformation in the bottom right triangle is the unit of the adjunction $(\pi^*, \pi_*)$ and $\alpha^*$ is induced by the square \[ \xymatrix{ S^\triangleleft \ar[d]_{\id} & S \dltwocell\omit{\alpha} \ar[d]^{\pi} \ar[l]_-{i_S} \\
S^\triangleleft & \base . \ar[l]^{-\infty}
}\] Note that $\dia{S} \circ \pi^*$ is the constant diagram functor $\Delta_S$, and so the vertical pasting of the triangles on the right is the diagonal map $Y \rightarrow \holim{S}{\Delta_S(Y)} = \prod_{s\in S} Y$ for each object $Y$ in $\Der(\base)$.  The vertical pasting of the squares on the left yields a natural transformation $\Delta_S(X_{-\infty}) \rightarrow i_S^*(\dia{S^\triangleleft}(X))$ induced by the structure maps of $X$.  The pasting of the entire diagram therefore gives rise to the map $X_{-\infty} \rightarrow \prod_{s\in S} X_s$ produced by the universal property of the product applied to $\dia{S^\triangleleft}(X)$.  So $\dia{S^\triangleleft}(X)$ is exhibiting $X_{-\infty}$ as the product if and only if this morphism is an isomorphism.  Since the top row is inhabited by invertible natural transformations, we have that the total pasting is a natural isomorphism whenever the pasting of the bottow row is a natural isomorphism.  By \cite[Prop.~2.6]{Groth}, this occurs exactly when $X$ is a limiting cone.
\end{proof}

\begin{proof}[Proof of Proposition \ref{Prop: reduced product diagram}]
We first define a small category $P(S)$ containing each proper filter on $S$ as a full subcategory and show that there exists $\tilde{X}$ in $\Der(P(S))$ satisfying the conditions of the theorem.  Later we will restrict to the filter $\mathcal{F}$ in particular.  

Let $P(S)$ be the small category with objects $\emptyset \neq P \in \Pow{S}$ and morphisms $f_{PQ} \colon P \rightarrow Q$ if and only if $Q \subseteq P$.  Consider the functor $l_S \colon S \rightarrow P(S)$ defined by $s \mapsto \{s\}$ and the right Kan extension $(l_S)_* \colon \Der(S) \rightarrow \Der(P(S))$ along $l_S$. For each $X$ in $\Der(S)$, define \[\tilde{X} := (l_S)_*(X).\] 

\noindent \emph{Step 1: Show that for each $P \in P(S)$, the value $\tilde{X}_P$ of $\tilde{X}$ at $P$ is isomorphic to $\prod_{p\in P} X_p$:}  Consider the slice square \[ \xymatrix{(P/l_S) \ar[r]^q \ar[d] & S \ar[d]^{l_S} \\ \base \ar[r]_{P} & P(S) \ultwocell\omit{} .} \]  By (Der4), the associated canonical mate transformation $P^*(l_S)_* \rightarrow \holimf{(P/l_S)} q^*$ is an isomorphism.  Note that the comma category $(P/l_S)$ is equivalent to the discrete category $P$ and $q$ is the canonical embedding of $P\subseteq S$.  By \cite[Prop.~1.7]{Der}, we have that \[\tilde{X}_P = (l_S)_*(X)_P \cong \holimf{(P/l_S)}q^*(X) \cong \prod_{p\in P}X_p.\]

\noindent \emph{Step 2: Show that the value $\tilde{X}_{f_{PQ}} \colon \tilde{X}_P \rightarrow \tilde{X}_Q$ of $\tilde{X}$ at $f_{PQ}$ is the canonical projection $\phi_{PQ}$ where $P$ is in $P(S)$ and $Q = \{p\}$ for $p \in P$:} Consider the fully faithful functor $v_P \colon P^\triangleleft \rightarrow P(S)$ where $p \mapsto \{p\}$ and $-\infty \mapsto P$.  By Lemma \ref{Lem: Moritz}, it suffices to show that $v_P^*\tilde{X}$ is a limiting cone.  Consider the square \[ \xymatrix{P \ar[r]^{i_P} \ar[d]_{j_P} \drtwocell\omit{id} & P^\triangleleft \ar[d]^{v_P} \\ S \ar[r]_{l_S} & P(S) } \] where $j_P$ is the embedding of $P$ into $S$.  If this square is homotopy exact then $v_P^*(l_S)_*(X) \cong (i_P)_*j_P^*(X)$ as desired.  The square can be expressed as the following vertical pasting \[ \xymatrix{P \ar[r]^{i_P} \ar[d]_{id} \drtwocell\omit{id} & P^\triangleleft \ar[d]^{v_P} \\ P \ar[r]_{l_P} \ar[d]_{j_P} \drtwocell\omit{id}  & P(P) \ar[d]^{j_{P(P)}} \\ S \ar[r]_{l_S} &P(S).} \]  By \cite[Lem.~2.12]{Groth}, the top square is homotopy exact and so, by \cite[Lem.~1.14]{Der}, it suffices to show that the bottom square is homotopy exact.  

Since $j_P$ and $j_{P(P)}$ are fully faithful, it follows from \cite[Lem.~2.12]{Groth} that it is enough to show that the canonical mate transformation $(j_P)_!l_P^* \rightarrow l_S^*(j_{P(P)})_!$ is an isomorphism for all $s \in S\setminus P$.  Let $Y$ be an object in $\Der(P(P))$ and note that $l_S^*(j_{P(P)})_!(Y)_s \cong (j_{P(P)})_!(Y)_{\{s\}}$.  The functors $j_{P(P)}$ and $j_P$ are both cosieves.  Since $\{s\}$ is not in the image of $j_{P(P)}$and $s$ is not in the image of $j_P$, it follows from \cite[Prop.~1.23]{Der} that both $(j_{P(P)})_!(Y)_{\{s\}}$ and $(j_P)_!l_P^*(Y)_s$ are isomorphic to initial objects in $\Der(\base)$.  Thus $(j_P)_!l_P^*(Y)_s \rightarrow l_S^*(j_{P(P)})_!(Y)_s$ is the unique isomorphism between initial objects.  It follows that the bottom square is homotopy exact as required.\\

\noindent \emph{Step 3: Show that $\tilde{X}_{f_{PQ}} \colon \tilde{X}_P \rightarrow \tilde{X}_Q$ is the canonical projection $\phi_{PQ}$ for each $Q \subseteq P$ in $P(S)$:} Let $k_{Q^\triangleleft} \colon (Q^\triangleleft)^\triangleleft \rightarrow P(S)$ be the functor defined by $q \mapsto \{q\}$ for all $q\in Q$, $-\infty \mapsto Q$ and $-\infty-1\mapsto P$.  Then, by Step 3, the underlying diagram $\dia{(Q^\triangleleft)^\triangleleft}(k_{Q^\triangleleft}^*\tilde{X})$ is isomorphic to the incoherent diagram consisting of commutative triangles \[\xymatrix{\prod_{ P} X_p \ar[r]^u \ar[dr]_{\pi_{P\{q\}}} & \prod_{ Q} X_q \ar[d]^{\pi_{Q\{q\}}} \\ & X_q }\]  for each $q\in Q$.  By the universal property of the product, the morphism $u$ must be the canonical projection.\\

\noindent \emph{Step 4: Restrict to $\mathcal{F}$:} Let $u \colon \mathcal{F} \rightarrow P(S)$ be the fully faithful functor mapping each $P \in \mathcal{F}$ to itself.  Then let $\Redf := u^*\circ (l_S)_*$.  It follows from the above steps that $\Red{X}$ has the desired properties.
\end{proof}


\bibliographystyle{abbrv}
\bibliography{CosiltHeart}
\end{document}